\documentclass[reqno,12pt]{amsart}
\usepackage{mathrsfs}
\usepackage{amscd}
\usepackage{amssymb}
\usepackage{amsmath}
\usepackage{amsfonts}
\usepackage{enumerate}
\usepackage{graphicx}
\usepackage{epic}
\usepackage[all,cmtip]{xy}
\usepackage{overpic}
\usepackage{subfig}
\usepackage{caption}
\usepackage{tikz}

\newtheorem{thm}{Theorem}[section]
\newtheorem{prop}[thm]{Proposition}
\newtheorem{lem}[thm]{Lemma}
\newtheorem{cor}[thm]{Corollary}
\theoremstyle{definition}
\newtheorem{Def}[thm]{Definition}
\newtheorem{conj}[thm]{Conjecture}
\newtheorem{exmp}[thm]{Example}

\newtheorem{rem}[thm]{Remark}
\newtheorem*{rem*}{Remark}

\newcommand{\kk}{\mathbf{k}}
\newcommand{\xx}{\mathbf{x}}

\newcommand{\ZZ}{\mathcal {Z}}
\newcommand{\Aa}{\mathcal {A}}
\newcommand{\WW}{\mathcal {W}}
\newcommand{\KK}{\mathcal {K}}

\newcommand{\Ll}{\mathcal {L}}

\newcommand{\CC}{\mathcal {C}}
\newcommand{\RR}{\mathcal {R}}
\newcommand{\Ss}{\mathcal {S}}

\newcommand{\VV}{\mathcal {V}}
\newcommand{\UU}{\mathcal {U}}

\newcommand{\FF}{\mathcal {F}}
\newcommand{\II}{\mathcal {I}}

\newcommand{\DD}{\mathcal {D}}
\newcommand{\Zz}{\mathbb {Z}}
\newcommand{\Cc}{\mathbb {C}}
\newcommand{\Rr}{\mathbb {R}}
\newcommand{\Qq}{\mathbb {Q}}
\newcommand{\Nn}{\mathbb {N}}

\def\w{\widetilde}

\def\xr{\xrightarrow}
\def\sfr{\smallfrown}
\def\ssm{\smallsmile}

\numberwithin{equation}{section}
\usepackage[colorlinks,linktocpage]{hyperref}
\hypersetup{citecolor=blue,linkcolor=blue}
\begin{document}
	\title[Toric spaces and face enumeration on simplicial manifolds]{Toric spaces and face enumeration on simplicial manifolds}
	\author[F.~Fan]{Feifei Fan}
	\thanks{The author is supported by the National Natural Science Foundation of China (Grant Nos. 11801580, 11871284)}
	\address{Feifei Fan, School of Mathematical Sciences, South China Normal University, Guangzhou, 510631, China.}
	\email{fanfeifei@mail.nankai.edu.cn}
	\subjclass[2010]{Primary 05E45, 13F55, 57P10; Secondary 05E40, 14M25, 55N32, 57P05.}
	\maketitle
	\begin{abstract}
		In this paper, we study the well-know $g$-conjecture for rational homology spheres in a topological way.
		To do this, we construct a class of topological spaces with torus actions, which can be viewed as topological generalizations of toric varieties.
		Along this way we prove that after doing stellar subdivision operations at certain middle dimensional faces of an arbitrary rational homology sphere, the $g$-conjecture is valid.
		Furthermore, we give topological proofs of several fundamental algebraic results about Buchsbaum complexes and simplicial manifolds.
		In this process, we also get a few interesting results in toric topology.
	\end{abstract}
	\tableofcontents
	\section{Introduction}
	\renewcommand{\thethm}{\arabic{thm}}
	Our motivating problem is the following celebrated $g$-conjecture in algebraic combinatorics, which was first
	proposed by McMullen for characterizing the face numbers of simplicial polytopes \cite{Mc71}.
	See the great survey article \cite{Sw14} about this conjecture by Swartz.
	\begin{conj}[$g$-conjecture]\label{conj:g-con}
		The $g$-vector of a rational homology sphere is a $M$-vector.
	\end{conj}
	
	To understand this conjecture, let us recall some notions.
	
	For a $(d-1)$-dimensional simplicial complex $\Delta$, the \emph{$f$-vector} of $\Delta$ is
	\[(f_0,f_1,\dots,f_{d-1}),\]
	where $f_i$ is the number of the $i$-dimensional faces of $\Delta$.
	Sometimes it is convenient to set $f_{-1}=1$ corresponding to the empty set. The \emph{$h$-vector} of $\Delta$ is the integer vector
	$(h_0,h_1,\dots,h_d)$ defined from the equation
	\[
	h_0t^d+\cdots+h_{d-1}t+h_d=f_{-1}(t-1)^d+f_0(t-1)^{d-1}+\cdots+f_{d-1}.
	\]
	The $g$-vector $(g_0,\dots,g_{\lfloor d/2 \rfloor})$ is defined to be $g_0=1$, $g_i=h_i-h_{i-1}$ for $1\leqslant i\leqslant \lfloor d/2 \rfloor$.
	
	In order to define $M$-vectors we first introduce the pseudopowers.
	For any two positive integers $a$ and $i$ there is a unique way to write
	\[a=\binom{a_i}{i}+\binom{a_{i-1}}{i-1}+\cdots+\binom{a_j}{j}\]
	with $a_{i}>a_{i-1}>\cdots>a_{j} \geqslant j \geqslant 1$. Define the $i$th \emph{pseudopower of $a$} as
	\[a^{\langle i\rangle}=\binom{a_i+1}{i+1}+\binom{a_{i-1}+1}{i}+\cdots+\binom{a_j+1}{j+1}.\]
	For convenience we define $0^{\langle i\rangle}=0$ for all $i$.
	A sequence of integers $(k_0,k_1,k_2,\dots)$ satisfies $k_{0}=1$ and  $0\leqslant k_{i+1} \leqslant k_{i}^{\langle i\rangle}$ for $i\geqslant1$ is called an \emph{$M$-sequence}. Finite $M$-sequences are \emph{$M$-vectors}. Its name comes from the following fundamental result of Macaulay.
	\begin{thm}[Macaulay {\cite{Mac27}}, see {\cite[\S 4.2]{BH98}}]\label{thm:m-vector}
		A sequence of integers $(k_0,k_1,k_2,\dots)$ is a $M$-sequence if and only if there exists a connected commutative graded algebra $A=A_0\oplus A_1\oplus A_2\oplus\cdots$ over a field $\kk$ such that $A$ is generated by its degree-one elements and $\dim_\kk A_i=k_i$ for $i\geqslant 0$.
	\end{thm}
	
	In 1980, by using results from algebraic geometry, Stanley gave a beautiful proof of Conjecture \ref{conj:g-con} for the case where $\Delta$ is a \emph{polytopal sphere}, i.e. the boundary complex of a simplicial polytope. Later, McMullen gave another proof of the same result without using algebraic geometry.
	
	Let $\Theta=(\theta_1,\dots,\theta_d)$ be a l.s.o.p. (linear system of parameters) of the face ring $\Qq[\Delta]$ of a $(d-1)$-compelx $\Delta$ (see \S \ref{subsec:l.s.o.p.} for the definitions).
	Stanley first noticed that when $\Delta$ is Cohen-Macaulay (a class of simplicial complexes including simplicial spheres), the Hilbert function of $\Qq[\Delta]/\Theta$ is equal to the $h$-vector of $\Delta$.
	So in the case of polytopal spheres, Conjecture \ref{conj:g-con} is an immediate consequence  of Theorem \ref{thm:m-vector} and the following theorem.
	\begin{thm}[Stanley \cite{S80}, McMullen \cite{Mc93,Mc96}]\label{thm:Lefschetz}
		If $\Delta$ is the boundary of a simplicial $d$-polytope, then for a certain l.s.o.p. $\Theta$ of $\Qq[\Delta]$, there exists a linear form $\omega\in\Qq[\Delta]$ such that the multiplication map
		\[\cdot\omega^{d-2i}:(\Qq[\Delta]/\Theta)_{i}\to(\Qq[\Delta]/\Theta)_{d-i}\]
		is an isomorphism for all $i\leqslant d/2$.
	\end{thm}
	Stanley's proof of the theorem above  used deep results from algebraic geometry, in particular, the hard Lefschetz theorem for projective toric varieties. McMullen's proof builds upon the notion of the \emph{polytope algebra}, which may be
	thought of as a combinatorial model for the cohomology algebras of toric varieties.
	
	Let $\Delta$ be a rational homology $(d-1)$-sphere. Then it satisfies the Dehn-Sommerville relations, i.e., $h_i(\Delta)=h_{d-i}(\Delta)$ \cite{Kle64}.
	We say $\Delta$ has \emph{Lefschetz property} if there exists an l.s.o.p. $\Theta$ for $\Qq[\Delta]$ and a linear form $\omega$ satisfying the condition in Theorem \ref{thm:Lefschetz}. Obviously the $g$-conjecture can be deduced from the following algebraic conjecture.
	\begin{conj}[algebraic $g$-conjecture]\label{conj:Lefschetz}
		Every rational homology sphere has Lefshetz property.
	\end{conj}
	Recently, Adiprasito \cite{A18} announced a proof of conjecture \ref{conj:Lefschetz}, but his paper is too technical and complicated.
	Still, the first three sections of his paper are inspiring and readable.
	
	In the spirit of Stanley's topological proof of Theorem \ref{thm:Lefschetz}, the following natural question arises.
	
	\emph{What is the topological spaces behind the $g$-conjecture for general simplicial spheres (or even rational homology spheres)?}
	
	In this paper, we answer this question by constructing a class of topological spaces with torus actions, as a generalization of toric varieties.
	It turns out that some algebraic properties of face rings can be explained by the topological properties of these toric spaces, such as the Dehn-Sommerville relations
	just correspond to the Poincar\'e duality of rational toric manifolds as we will show in \S \ref{sec:duality}.
	
	We can deduce many interesting results from the well-behaved local topology of these toric spaces. For example, in \S \ref{sec:stellar sub} we prove that after doing stellar subdivision operations at certain middle dimensional faces of an arbitrary rational homology sphere, the $g$-conjecture is valid (Corollary \ref{cor:g-conj}).
	
	Another important research object in algebraic combinatorics is the class of Buchsbaum complexes, which includes homology manifolds. In \S \ref{sec:buchsbaum complex}, we calculate the rational cohomology of toric spaces associated to a Buchsbaum complex $\Delta$ (Theorem \ref{thm:coho of toric over buchs}), especially when $\Delta$ is a rational homology manifold (Theorem \ref{thm:poincare duality}).
	This gives topological expositions for several fundamental algebraic results about Buchsbaum complexes (e.g. Theorem \ref{thm:schenzel} and Theorem \ref{thm:socle of manifolds}).
	
	Our topological construction is inspired by Davis-Januszkiewicz's  \cite{DJ91} construction of \emph{(quasi)toric manifolds} over simple polytopes.
	Their pioneering work \cite{DJ91} is the beginning of a very recent field called \emph{toric topology}.
	\S \ref{subsec:m-a}-\ref{subsec:cohom toric} are short introductions to the main two research objects in toric topology: \emph{moment-angle complexes} and \emph{toric spaces by D-J construction}.
	
	\section{Preliminaries}
	\renewcommand{\thethm}{\thesection.\arabic{thm}}
	\subsection{Notations and conventions}\label{subsec:notation}
	
	For an abstract simplicial complex $\Delta$, let $\FF_i(\Delta)$ be the set of $i$-dimensional faces (simplices) of $\Delta$. For convenience, we set $\FF_{-1}=\{\varnothing\}$. Unless otherwise stated, we assume $\Delta$ has $m$ vertices and identify $\FF_0(\Delta)$ with $[m]=\{1,\dots,m\}$.
	By $\Delta^{m-1}$ we denote the simplex consisting of all subsets of $[m]$, and by $\partial\Delta^{m-1}$ the boundary complex of $\Delta^{m-1}$.
	
	For a subset $J\subset[m]$, the \emph{full subcomplex} $\Delta_J\subset\Delta$ is defined to be
	\[\Delta_J=\{\sigma\in\Delta:\sigma\subset J\}.\]
	
	A subset $I\subset[m]$ is a \emph{missing face} of $\Delta$ if $I\not\in\Delta$ but $J\in\Delta$ for all proper subsets $J\subset I$.
	
	The \emph{link} and the \emph{star} of a face $\sigma\in\Delta$ are the subcomplexes
	\[\begin{split}
		\mathrm{lk}_\sigma\Delta=&\{\tau\in\Delta:\tau\cup\sigma\in\Delta,\tau\cap\sigma=\varnothing\};\\
		\mathrm{st}_\sigma\Delta=&\{\tau\in\Delta:\tau\cup\sigma\in\Delta\}.
	\end{split}\]
	
	The \emph{join} of two simplicial complexes $\Delta$ and $\Delta'$, where the vertex set $\FF_0(\Delta)$ is disjoint from $\FF_0(\Delta')$ , is the simplicial complex
	\[\Delta*\Delta'=\{\sigma\cup\sigma':\sigma\in\Delta, \sigma'\in\Delta'\}.\]
	In particular, we say that $\Delta^0*\Delta$ is the \emph{cone} over $\Delta$, simply denoted $\CC\Delta$.
	
	Let $\sigma\in \Delta$ be a nonempty face of $\Delta$.
	The \emph{stellar subdivision} of $\Delta$ at $\sigma$ is obtained by replacing the star of $\sigma$ by the cone
	over its boundary:
	\[\Delta(\sigma)=(\Delta\setminus \mathrm{st}_\sigma\Delta)\cup \big(\CC(\partial\sigma*\mathrm{lk}_\sigma\Delta)\big).\]
	If $\dim\sigma=0$ then $\Delta(\sigma)=\Delta$. Otherwise the complex $\Delta(\sigma)$ acquires an additional vertex (the apex of the cone). In this case, denote by $v_\sigma$ this new vertex.
	
	If $\kk$ is a field, the \emph{reduced Betti numbers} of $\Delta$ are $\w \beta_i(\Delta;\kk):=\dim_\kk\w H_i(\Delta;\kk)$.
	
	A simplicial complex $\Delta$ is called a \emph{triangulated manifold} (or \emph{simplicial manifold}) if the geometric realization $|\Delta|$ is a topological manifold. More generally, a $d$-dimensional simplicial complex $\Delta$ is a \emph{$\kk$-homology manifold} ($\kk$ is a commutative ring) if
	\[H_*(|\Delta|,|\Delta|-x;\kk)=\w H_*(S^{d};\kk)\quad \text{for all }x\in|\Delta|,\]
	or equivalently, \[H_*(\mathrm{lk}_\sigma\Delta;\kk)=H_*(S^{d-|\sigma|};\kk)\quad \text{for all }\varnothing\neq\sigma\in\Delta.\]
	Especially, when $\kk=\Qq$, it is also referred to as a \emph{rational homology manifold}, and when $\kk=\Zz$, it is simply called a \emph{homology manifold}.
	The notions for manifold, such as orientable, closed, with boundary, etc., are similarly defined for $\kk$-homology manifold. For example, A pair $(\Delta,\partial \Delta)$ of simplicial complexes is a \emph{$\kk$-homology $d$-manifold with boundary} if the following conditions hold:
	\begin{itemize}
		\item $\Delta-\partial\Delta$ is a $\kk$-homology $d$-manifold,
		\item $\partial\Delta$ is a $\kk$-homology $(d-1)$-manifold, and
		\item for each $x\in|\partial\Delta|$, the  homology groups $H_*(|\Delta|,|\Delta|-x;\kk)$ all vanish.
	\end{itemize}
	
	$\Delta$ is a \emph{$\kk$-homology $d$-sphere} if it is a $\kk$-homology $d$-manifold with the same $\kk$-homology as $S^d$.
	Similarly, when $\kk=\Qq$, it is also called a \emph{rational homology sphere}, and for $\kk=\Zz$, a \emph{homology sphere}.
	(Remark: Usually, the terminology ``homology sphere"  means a manifold having the homology of a sphere. Here we take it in a more relaxed sense than its usual meaning.) The \emph{$\kk$-homology ball} is defined similarly.
	\subsection{Face rings and l.s.o.p}\label{subsec:l.s.o.p.}
	For a commutative ring $\kk$ with unit, let $\kk[x_1,\dots,x_m]$ be the polynomial algebra with one generator for each
	vertex in $\Delta$. We make it a graded algebra by setting $\mathrm{deg}\,x_i=2$.
	(This even grading is unusual for algebraists. The reason why we set $\mathrm{deg}\,x_i=2$ rather than $1$ is to make it agree with the grading of the cohomology of some toric spaces we constructed below.)
	
	The \emph{Stanley-Reisner ideal} of $\Delta$ is
	\[I_\Delta:=(x_{i_1}x_{i_2}\cdots x_{i_k}:\{i_1,i_2,\dots,i_k\}\not\in\Delta).
	\]
	The \emph{Stanley-Reisner ring} (or \emph{face ring}) of $\Delta$ is the quotient \[\kk[\Delta]:=\kk[x_1,\dots,x_m]/I_\Delta.\]
	Since $I_\Delta$ is a monomial ideal, the quotient ring $\kk[\Delta]$ is graded by degree.
	
	For a face $\sigma=\{x_1,\dots,i_k\}\in\FF_{k-1}(\Delta)$, denote by $\xx_\sigma=x_{i_1}\cdots x_{i_k}\in\Qq[\Delta]$ the face monomial corresponding to $\sigma$.
	
	Assuming $\kk$ is a field, a set $\Theta=\{\theta_1,\dots,\theta_d\}$ consisting of $d=\textrm{dim}\Delta+1$ linear forms in $\kk[\Delta]$ is called a \emph{linear
		system of parameters} (l.s.o.p. for short), if $\kk[\Delta]/\Theta$ is finite-dimensional as a vector space over $\kk$; here $\Theta:=(\theta_1,\dots,\theta_d)$ also denotes the ideal that the l.s.o.p generates.
	It can be shown that a linear sequence $\theta_1,\dots,\theta_d$ is an l.s.o.p if and only if the restriction $\Theta_\sigma=r_\sigma(\Theta)$ to each face $\sigma\in\Delta$ generates the polynomial algebra $\kk[x_i:i\in\sigma]$; here $r_\sigma:\kk[\Delta]\to\kk[x_i:i\in\sigma]$ is the projection homomorphism (see \cite[Theorem 5.1.16]{BH98}). For the case that $\kk=\Zz$, a linear sequence $\theta_1,\dots,\theta_d$  is referred to as an \emph{integral l.s.o.p}
	if its reduction modulo $p$ is an l.s.o.p. for $\Zz_p[\Delta]$ for any prime $p$.
	Equivalently, $\theta_1,\dots,\theta_d$ is an integral l.s.o.p. if and only if the restriction $\Theta_\sigma=r_\sigma(\Theta)$ to each simplex $\sigma\in\Delta$ generates the polynomial algebra $\Zz[x_i:i\in\sigma]$.
	
	\begin{rem}
		If $\kk$ is an infinite field, then there always exists an l.s.o.p for $\kk[\Delta]$ by Noether normalization lemma, but if $\kk$ is a finite field (or $\kk=\Zz$) then an l.s.o.p. for $\kk[\Delta]$ (or an integral l.s.o.p. for $\Zz[\Delta]$) may fail to exist  (cf. \cite[Example 3.3.4]{BP15}).
	\end{rem}

	\subsection{Algebraic properties of face rings} In this subsection we review some basic combinatorial and algebraic concepts used in the rest of our paper. Throughout this subsection, $\kk$ is an infinite field of arbitrary characteristic.
	
	Let $\Delta$ be a simplicial complex of dimension $d-1$. The face ring $\kk[\Delta]$ is a \emph{Cohen-Macaulay ring} if for any l.s.o.p $\Theta=\{\theta_1,\dots,\theta_d\}$, $\kk[\Delta]/\Theta$ is a free $\kk[\theta_1,\cdots,\theta_d]$ module. In this case, $\Delta$ is called a \emph{Cohen-Macaulay complex over $\kk$}.
	
	Let $A$ be a connected commutative graded $\kk$-algebra. The \emph{socle} of $A$ is the ideal
	\[\mathrm{Soc}(A)=\{x\in A:A_+\cdot x=0\}.\]
	The face ring $\kk[\Delta]$ is a \emph{Gorenstein ring} if it is Cohen-Macaulay and for any l.s.o.p $\Theta=\{\theta_1,\dots,\theta_d\}$, $\dim_\kk\mathrm{Soc}(\kk[\Delta]/\Theta)=1$.
	In other words, $\kk[\Delta]/\Theta$ is a Poincar\'e duality $\kk$-algebra.
	We call $\Delta$ \emph{Gorenstein over $\kk$} if its face ring $\kk[\Delta]$ is a Gorenstein ring.
	Further, $\Delta$ is called \emph{Gorenstein*} if $\kk[\Delta]$ is Gorenstein and $\Delta$ is not a cone, i.e., $\Delta\neq\Delta^0*\Delta'$.
	
	The face ring $\kk[\Delta]$ is said to be \emph{Buchsbaum} if for every l.s.o.p $\{\theta_1,\dots,\theta_d\}$ and all $1\leqslant i\leqslant d$,
	\[\{x\in \kk[\Delta]/(\theta_1,\dots,\theta_{i-1}):x\theta_i=0\}=\mathrm{Soc}(\kk[\Delta]/(\theta_1,\dots,\theta_{i-1})).\]
	Similarly, $\Delta$ is called \emph{Buchsbaum over $\kk$} in this case.
	
	All these algebraic properties of face rings have combinatorial-topological characterisations as follows.
	\begin{thm}\label{thm:algebraic property}
		Let $\Delta$ be a simplicial complex. Then
		\begin{enumerate}[(a)]
			\item {\rm(Reisner \cite{Rei76})} $\Delta$ is Cohen-Macaulay (over $\kk$) if and only if for all faces $\sigma\in\Delta$ (including $\sigma=\varnothing$)
			and $i<\dim\mathrm{lk}_\sigma\Delta$, we have $\w H_i(\mathrm{lk}_\sigma\Delta;\kk)=0$.\vspace{8pt}
			
			\item {\rm(Stanley \cite[Theorem II.5.1]{S96})} $\Delta$ is Gorenstein* (over $\kk$) if and only if it is a $\kk$-homology sphere.\vspace{8pt}
			
			\item\label{item:3} {\rm(Schenzel \cite{Sch81})} $\Delta$ is Buchsbaum (over $\kk$) if and only if it is pure and the link of each nonempty face is Cohen-Macaulay (over $\kk$).
		\end{enumerate}
	\end{thm}
	Hence, every simplicial complex whose geometric realization is a $\kk$-homology manifold is Buchsbaum over $\kk$.
	
	If $\Delta$ is Cohen-Macaulay, the following result of Stanley shows that the $h$-vector of $\Delta$ has a pure algebraic description.
	\begin{thm}[Stanley]\label{thm:stanley}
		Let $\Delta$ be a $(d-1)$-dimensional Cohen-Macaulay complex and let $\Theta=\{\theta_1,\dots,\theta_d\}$ be an l.s.o.p. for $\kk[\Delta]$. Then
		\[\dim_\kk(\kk[\Delta]/\Theta)_{2i}=h_i(\Delta),\quad \text{for all } 0\leqslant i\leqslant d.\]
	\end{thm}
	
	A generalization of Theorem \ref{thm:stanley} for Buchsbaum complexes was found by Schenzel \cite{Sch81}.
	\begin{thm}[Schenzel \cite{Sch81}]\label{thm:schenzel}
		Let $\Delta$ be a $(d-1)$-dimensional Buchsbaum complex and let $\Theta=\{\theta_1,\dots,\theta_d\}$ be an l.s.o.p. for $\kk[\Delta]$. Then for all $0\leqslant j\leqslant d$,
		\[\dim_\kk(\kk[\Delta]/\Theta)_{2j}=h_j(\Delta)-{d\choose j}\sum_{i=1}^{j-1}(-1)^i\w\beta_{j-i-1}(\Delta;\kk).\]
	\end{thm}
	In \S \ref{subsec:buchsbaum}, we give a topological exposition of this formula in characteristic zero.
	
	\begin{Def}\label{def:WLP}
		Let $\Delta$ be a simplicial complex of dimension $d-1$. We say that $\Delta$  has the \emph{weak Lefschetz property over $\kk$} (WLP for short)
		if there is an l.s.o.p. $\{\theta_1,\dots,\theta_d\}$ for $\kk[\Delta]$ and a linear form
		$\omega$ such that the multiplication maps
		\[\cdot\omega:(\kk[\Delta]/\Theta)_{2i}\to(\kk[\Delta]/\Theta)_{2i+2}\]
		have full rank for all $i<d$, i.e. either
		injective or surjective.  Such a linear form $\omega$ is called a \emph{weak Lefschetz element} (WLE).
	\end{Def}
	The WLP is closely related to the $g$-conjecture because of the following well known result (cf. \cite{Sw14}).
	\begin{prop}\label{prop:WLP}
		Let $\Delta$ be a $\kk$-homology $(d-1)$-sphere. If $\Delta$ has the WLE over $\kk$ then $\Delta$ satisfies the $g$-conjecture;
		A linear form $\omega$ is a WLE if and only if the multiplication map
		$(\kk[\Delta]/\Theta)_{2\lfloor d/2 \rfloor}\xr{\cdot\omega}(\kk[\Delta]/\Theta)_{2\lfloor d/2 \rfloor+2}$ is surjective, or equivalently, the map $(\kk[\Delta]/\Theta)_{2\lceil d/2 \rceil-2}\xr{\cdot\omega}(\kk[\Delta]/\Theta)_{2\lceil d/2 \rceil}$ is injective.
	\end{prop}
	
	We define a set of pairs $\WW(\Delta)\subset \kk^{f_0}\oplus\kk^{df_0}$ to be
	\[\WW(\Delta)=\{(\omega,\Theta):\Theta\text{ is an l.s.o.p. for }\kk[\Delta]\text{ and }\omega \text{ is a WLE }\}.\]
	It is well known that $\WW(\Delta)$ is a Zariski open set (see e.g. \cite[Proposition 3.6]{Swa06}).
	We will loosely use the term \emph{`generic choice'} of $\Theta$ or $\omega$ to mean that these elements are chosen from a
	non-empty Zariski open set, to be understood from the context.
	
	If $\kk$ is an infinite field and $\omega$ is a WEL for $\kk[\Delta]/\Theta$,
	then the generic linear combination of $\omega$ and some other arbitrary one-forms $\omega_1,\dots,\omega_k$ is also a WLE.
	This can be seen from the following elementary result in linear algebra theory.
	\begin{lem}\label{lem:generic}
		Suppose we are given $r\times s$ matrices ($r\leqslant s$) $A_1,\dots,A_j$ with entries in an infinite field $\kk$, and one of these matrices has rank $r$.
		Let $B_{b_1,\dots,b_j}=\sum_{i=1}^j b_iA_i$ ($b_i\in\kk$) be a  linear combination of $A_i,\dots,A_j$. Then the set \[X=\{(b_1,\dots,b_j)\in\kk^j:\mathrm{rank}\,B_{b_1,\dots,b_j}=r\}\] is a nonempty Zariski open subset in $\kk^j$.
	\end{lem}
	\begin{proof}
		Without loss of generality, we may assume $A_i,\dots,A_k$ are square $r\times r$ matrices, and $|A_1|\neq0$. Viewing $b_i$ as variables, then it is easily verified that the determinant $|B_{b_1,\dots,b_k}|$ is a nonzero homogeneous polynomial $f(b_1,\dots,b_k)$ of degree $r$. The statement of this lemma follows immediately since $\kk$ is infinite.
	\end{proof}
	\subsection{Moment-angle complexes and manifolds}\label{subsec:m-a}
	The moment-angle complexes first appeared in the work of Davis and Januszkiewicz \cite{DJ91} and further studied in detail and named by Buchstaber and Panov \cite{BP00}.
	They play a key role in the emerging field of toric topology, which has many connections with algebraic geometry, commutative algebra and combinatorics, etc.
	
	Let $\Delta$ be a simplicial complex, and let $(D^2, S^1)$ denote the pair of a disk and its boundary circle.
	For each simplex $\sigma=\{i_1,\dots,i_k\}\in \Delta$, set
	\[B_\sigma=\{(z_1,\dots,z_m)\in (D^2)^m:z_i\in S^1 \text{ when } i\not\in\sigma\}.\]
	The \emph{moment-angle complex} $\ZZ_\Delta$ is defined to be the CW complex
	\[\ZZ_\Delta:=\bigcup_{\sigma\in\Delta}B_\sigma\subset(D^2)^m.\]
	The standard coordinatewise action of the $m$-torus $T^m=\mathbb{R}^m/\mathbb{Z}^m$ on $(D^2)^m$ induces
	the canonical $T^m$-action on $\ZZ_\Delta$.
	
	We have a natural cellular decomposition of $\ZZ_\Delta$ as follows.
	Consider the following decomposition of the disc $D^2$ into 3 cells: the $0$-cell $e^0=1\in D^2$; the $1$-cell $e^1=S^1\setminus\{1\}$; the $2$-cell $e^2=D^2\setminus S^1$. By taking product we obtain a cellular decomposition of $(D^2)^m$, and then $\ZZ_\Delta$ embeds as a CW subcomplex
	in $(D^2)^m$. Each cell of $\ZZ_\Delta$ has the form
	\[e_\sigma\times t_J=e^2_{i_1}\times\cdots\times e^2_{i_k}\times e^1_{j_1}\times\cdots\times e^1_{j_r},\]
	where $\sigma=\{i_1,\dots, i_k\}\in\Delta$, $J=\{j_1,\dots, j_r\}$, $\sigma\cap J=\varnothing$. (We omit the $0$-cell in the product.)
	
	There is an alternative way to define $\ZZ_\Delta$ in terms of the dual simple polyhedral complex $P_\Delta$, constructed in \cite{DJ91}.
	As a polyhedron, $P_\Delta$ is the cone over the barycentric subdivision $\Delta'$ of $\Delta$.
	Precisely, for each simplex $\sigma\in \Delta$ (including $\varnothing$), let $F_\sigma$ denote the geometric realization of the poset
	$\Delta_{\geqslant\sigma}=\{\tau\in \Delta:\tau\geqslant\sigma\}$.
	Hence, for $\sigma\neq\varnothing$, $F_\sigma$ is the subcomplex of $\Delta'$ consisting of all
	simplices of the form $\sigma=\sigma_0<\sigma_1<\cdots<\sigma_k$, and $F_\varnothing=P_\Delta$ is the cone on $\Delta'$.
	If $\sigma$ is a $(k-1)$-simplex, then we say that $F_\sigma$ is a \emph{face of codimension $k$}.
	
	The polyhedron $P_\Delta$ together with its decomposition into ``faces" $\{F_\sigma\}_{\sigma\in \Delta}$ will be called a \emph{simple polyhedral complex}.
	In particular, there are $m$ facets $F_i,\dots,F_m$ of $P_\Delta$, in which $F_i$ is the geometric realization of the star of the $i$th vertex of $\Delta$ in $\Delta'$. Let $T_i=S^1$ be the coordinate circle subgroup of $T^m$.
	For each point $x\in P_\Delta$, define a subtorus
	\[T(x)=\prod_{i:\,x\in F_i}T_i\subset T^m,\]
	assuming that $T(x)=\{1\}$ if there are no facets containing $x$. Then define
	\begin{equation}\label{eq:m-a}
		\ZZ_\Delta=P_\Delta\times T^m/\sim,
	\end{equation}
	where the equivalence relation $\sim$ is given by $(x,g)\sim(x',g')$ if and only if $x=x'$ and $g^{-1}g'\in T(x)$.
	The action of $T^m$ on $P_\Delta\times T^m$ by the right translations descends to
	a $T^m$-action on $\ZZ_\Delta$, and the orbit space of this action is just $P_\Delta$.
	These two definitions are equivalent and both have their own convenience in different situations (cf. \cite[Chapter 6]{BP02}).
	\begin{exmp}\label{exmp:m-a}
		(i) Let $\Delta=\partial\Delta^{m-1}$ (the boundary of a simplex), then
		\[\begin{split}
			\ZZ_\Delta&=(D^2\times\cdots\times D^2\times S^1)\cup (D^2\times\cdots\times S^1\times D^2)\cup\cdots\\
			&\cup( S^1\times\cdots\times D^2\times D^2)=\partial\big((D^2)^m\big)=S^{2m-1}.
		\end{split}\]
		
		(ii) If $\Delta=\Delta_1*\Delta_2$, then $\ZZ_\Delta=\ZZ_{\Delta_1}\times\ZZ_{\Delta_2}$.
	\end{exmp}
	$\ZZ_\Delta$ is a closed orientable topological manifold (resp. $\kk$-homology manifold) of dimension $m+d$ if and only if $\Delta$ is a homology $(d-1)$-sphere (resp. $\kk$-homology $(d-1)$-sphere). (see \cite[\S2.1]{C17}). In this case,  we call $\ZZ_\Delta$ a \emph{moment-angle manifold} (resp. \emph{$\kk$-homology moment-angle manifold}).
	In particular, if $\Delta$ is a polytopal sphere, then $\ZZ_\Delta$ admits a smooth structure (see \cite[Chapter 6]{BP15}). In general, the smoothness of $\ZZ_\Delta$ is open.
	
	\subsection{Cohomology of moment-angle complexes}\label{subsec:cohom m-a}
	Throughout this subsection, $\kk$ is an commutative ring with unit.
	
	For a simplicial complex $\Delta$, the \emph{Koszul complex} of the face ring $\kk[\Delta]$ is defined as
	the differential $\Zz\oplus\Nn^m$-graded algebra $(\Lambda[y_1,\dots,y_m]\otimes\kk[\Delta],d)$,
	where $\Lambda[y_1,\dots,y_m]$ is the exterior algebra on $m$ generators over $\kk$, and the multigrading and differential is given by
	\begin{gather*}
		\mathrm{mdeg}\,y_i=(-1,2\boldsymbol{e}_i),\ \mathrm{mdeg}\,x_i=(0,2\boldsymbol{e}_i),\ \boldsymbol{e}_i\in\Nn^m\text{ is the $i$th unit vector };\\
		dy_i=x_i,\quad dx_i=0.
	\end{gather*}
	It is known that $H^*(\Lambda[y_1,\dots,y_m]\otimes\kk[\Delta],d)=\mathrm{Tor}_{\kk[x_1,\dots,x_m]}(\kk[\Delta],\kk)$.
	Then the Tor-algebra $\mathrm{Tor}_{\kk[x_1,\dots,x_m]}(\kk[\Delta],\kk)$ is canonically an $\Zz\oplus\Nn^m$-graded algebra.
	
	\begin{thm}[\cite{BBP04},{\cite[Theorem 4.5.4]{BP15}}]
		The following isomorphism of algebras holds:
		\begin{gather*}
			H^*(\ZZ_\Delta;\kk)\cong\mathrm{Tor}_{\kk[x_1,\dots,x_m]}(\kk[\Delta],\kk),\\
			H^p(\ZZ_\Delta;\kk)=\bigoplus_{-i+2|J|=p}\mathrm{Tor}^{-i,2J}_{\kk[x_1,\dots,x_m]}(\kk[\Delta],\kk),
		\end{gather*}
		where $J=(j_1,\dots,j_m)\in\Nn^m$ and $|J|=j_1+\cdots+j_m$.
	\end{thm}
	We may view a subset $J\subset[m]$ as a $(0,1)$-vector in $\Nn^m$ whose $j$th coordinate
	is $1$ if $j\in J$ and is $0$ otherwise. Then there is the following well known
	Hochster's formula:
	\begin{thm}[Hochster \cite{H75}, see also {\cite[Theorem 3.2.9]{BP15}}]
		For any subset $J\subset[m]$ we have
		\[\mathrm{Tor}^{-i,2J}_{\kk[x_1,\dots,x_m]}(\kk[\Delta],\kk)\cong\w H^{|J|-i-1}(\Delta_J;\kk),\]
		and $\mathrm{Tor}^{-i,2J}_{\kk[x_1,\dots,x_m]}(\kk[\Delta],\kk)=0$ if $J$ is not a $(0,1)$-vector. We assume $\w H^{-1}(\Delta_\varnothing;\kk)=\kk$ above.
	\end{thm}
	
	So $\mathrm{Tor}_{\kk[x_1,\dots,x_m]}(\kk[\Delta],\kk)$ is isomorphic to $\bigoplus_{J\subset[m]}\w H^*(\Delta_J;\kk)$  as $\kk$-modules,
	and this isomorphism endows the direct sum $\bigoplus_{J\subset[m]}\w H^*(\Delta_J;\kk)$ with a $\kk$-algebra structure.
	On the other hand, Baskakov \cite{B02} directly defined a multiplication structure on $\bigoplus_{J\subset[m]}\w H^*(\Delta_J;\kk)$ to make the isomorphism in the Hochster's formula
	to be algebraic.
	Before describing this multiplication structure precisely, let us see some operations on the homology and cohomology of the full subcomplexes of $\Delta$.
	
	Let $\w C^i(\Delta;\kk)$ (resp. $\w C_i(\Delta;\kk)$) denote the $i$th reduced simplicial cochain (resp. chain) group of $\Delta$ with coefficients in $\kk$. For an oriented (ordered) simplex $\sigma=(i_1,\dots,i_p)\in \Delta$, denote still by $\sigma\in\w C^{p-1}(\Delta;\kk)$ the basis cochain corresponding to $\sigma$; it takes value $1$ on $\sigma$ and vanishes on all other simplices.  For simplicity we will omit the coefficient ring $\kk$ from the notations throughout the rest of this subsection.
	\begin{Def}\label{def:cup}
		The \emph{union product} in the simplicial cochains of full subcomplexes of $\Delta$ is defined to be the $\kk$-bilinear operation
		\begin{align*}
			\sqcup :\w C^{p-1}(\Delta_I)\otimes \w C^{q-1}(\Delta_J)&\to \w C^{p+q-1}(\Delta_{I\cup J}), \quad p,q\geqslant 0,\\
			\sigma\otimes\tau&\mapsto\sigma\sqcup\tau
		\end{align*}
		in which $\sigma\sqcup \tau$ is the juxtaposition of $\sigma$ and $\tau$ if $I\cap J=\varnothing$ and $\sigma\cup\tau$ is a simplex of $\Delta_{I\cup J}$; zero otherwise.
		
		Similarly, the \emph{excision product} in the simplicial chains and cochains of full subcomplexes is defined by
		\begin{align*}
			\sqcap :\w C_{p+q-1}(\Delta_I)\otimes \w C^{p-1}(\Delta_J)&\to \w C_{q-1}(\Delta_{I\setminus J}), \quad p,q\geqslant 0.\\
			\sigma\otimes\tau&\mapsto\sigma\sqcap\tau
		\end{align*}
		Here $\sigma\sqcap\tau=\varepsilon_{\sigma,\tau}(\sigma\setminus\tau)$ if $J\subset I$, $\tau\subset\sigma$ and $\sigma\setminus\tau\subset I\setminus J$; zero otherwise, and $\varepsilon_{\sigma,\tau}$ is the sign of the permutation sending $\tau\sqcup(\sigma\setminus\tau)$ to $\sigma$.
	\end{Def}
	It is easily verified that the union product of cochains
	induces a union product of cohomology classes in the full subcomplexes  of $K$:
	\begin{equation}\label{eq:union pruduct}
		\sqcup :\w H^{p-1}(\Delta)\otimes \w H^{q-1}(\Delta)\to \w H^{p+q-1}(\Delta), \quad p,q\geqslant 0.
	\end{equation}
	Similarly, there is an induced excision product in homology and cohomology of the full subcomplexes of $\Delta$. Union and excision product are related by the formula \[\psi(c\sqcap\phi)=(\phi\sqcup\psi)(c)\] for $c\in\w C_{p+q-1}(\Delta_I)$, $\phi\in\w C^{p-1}(\Delta_J)$ and $\psi\in\w C^{q-1}(\Delta_{I\setminus J})$.
	
	Intuitively, the union product (resp. excision product) is an analog of cup product (resp. cap product) in cohomology (resp. homology and cohomology) of a space.
	Actually, the union and excision product for $\Delta$ do induce the cup and cap product for $\ZZ_\Delta$, respectively (cf. \cite[Chapter 4.5]{BP15} and \cite{FW15}).
	
	\begin{thm}[{\cite[Proposition 3.2.10]{BP15}}]\label{thm:union product}
		There is a ring isomorphism (up to a sign for products).
		\[H^*(\ZZ_\Delta)\cong \bigoplus_{J\subset[m]}\w H^*(\Delta_J),\quad H^p(\ZZ_\Delta)\cong\bigoplus_{J\subset[m]}\w H^{p-|J|-1}(\Delta_J),\]
		where the ring structure on the right hand side is given by the union product $\sqcup$ in \eqref{eq:union pruduct}.
	\end{thm}
	
	\begin{rem}\label{rem:cell}
		Here is the  topological intuition for Theorem \ref{thm:union product}. For a subset $J\subset[m]$ and a $(k-1)$-face $\sigma\in \Delta_J$, we have a
		$(|J|+k)$-cell $e_\sigma\times t_{J\setminus\sigma}$ of $\ZZ_\Delta$.
		Let $C^*(\ZZ_\Delta)$ be the cellular cochain groups of $\ZZ_\Delta$. It has a basis of cochains $e_\sigma^* t_{J\setminus\sigma}^*$ dual to the corresponding cells.
		Hence, the cup product in $C^*(\ZZ_\Delta)$ just corresponds to the dual cochain of the cartesian product of these cells, so that corresponds up to a sign to the union product in the simplicial cochains of full subcomplexes of $\Delta$. Similarly, we have the correspondence between the cap product for $\ZZ_\Delta$ and excision product for the full subcomplexes of $\Delta$.
	\end{rem}
	To conclude this subsection, we mention that the $T^m$-equivariant cohomology is considerably simpler than the ordinary cohomology of $\ZZ_\Delta$.
	\begin{thm}[{\cite[Theorem 4.8]{DJ91}}]\label{thm:equiv}
		The $T^m$-equivariant cohomology ring of the moment-angle complex $\ZZ_\Delta$ is isomorphic to the face ring of $\Delta$:
		\[H^*_{T^m}(\ZZ_\Delta)\cong \kk[\Delta].\]
	\end{thm}
	\subsection{Quasitoric manifolds}
	In their pioneering work \cite{DJ91} Davis and Januszkiewicz suggested a topological generalisation of \emph{projective toric manifolds} (nonsingular projective toric varieties), which became
	known as \emph{quasitoric manifolds}. A quasitoric manifold is a $2d$-dimensional manifold $M$ with a locally standard
	action of $T^d$ (that is, it locally looks like the standard coordinatewise action of $T^d$ on $\Cc^d$) such that the quotient $M/T^d$ can be identified with a simple $d$-polytope $P$. Let us review this object as a guide to the further generalized spaces.
	
	Let $P$ be a simple $d$-polytope, $\FF=\{F_1,\dots,F_m\}$ the set of facets of $P$. Given a map $\lambda: \FF\to\mathbb{Z}^d$, and write $\lambda(F_i)$
	in the standard basis of $\mathbb{Z}^d$:
	\[\lambda(F_i)=\boldsymbol{\lambda}_i=(\lambda_{1i},\dots,\lambda_{di})^T\in\mathbb{Z}^d,\quad 1\leqslant i\leqslant m.\]
	If the matrix
	\[\mathit{\Lambda}=
	\begin{pmatrix}
		\lambda_{11}&\cdots&\lambda_{1m}\\
		\vdots&\ddots&\vdots\\
		\lambda_{d1}&\cdots&\lambda_{dm}
	\end{pmatrix}
	\]
	has the following property:
	\begin{equation}
		\mathrm{det}(\boldsymbol{\lambda}_{i_1},\dots,\boldsymbol{\lambda}_{i_d})=\pm1\quad \text{whenever }F_{i_1}\cap\dots\cap F_{i_d}\neq\varnothing\text{ in }P,
	\end{equation}
	then $\lambda$ is called a \emph{characteristic function} for $P$, and $\mathit{\Lambda}$ is called a \emph{characteristic matrix}.
	
	Let $(P,\mathit{\Lambda})$ be a \emph{characteristic pair} consisting of a simple polytope $P$ and its characteristic matrix $\mathit{\Lambda}$.
	Denote by $T_i=S^1$ the circle subgroup of $T^d=\mathbb{R}^d/\mathbb{Z}^d$ corresponding to the subgroup $\boldsymbol{\lambda_i}\in \mathbb{Z}^d$.
	For each point $x\in P$, define a subtorus
	\[T(x)=\prod_{i:\,x\in F_i}T_i\subset T^d.\]
	Then the \emph{quasitoric manifold} $M(P,\mathit{\Lambda})$ is defined to be
	\begin{equation*}
		M(P,\mathit{\Lambda})=P\times T^d/\sim,\quad\text{the relation $\sim$ is as in \eqref{eq:m-a}.}
	\end{equation*}
	In particular, If $P$ is a \emph{Delzant polytope} (a simple $d$-polytope $P\subset\Rr^d$ is called a Delzant polytope if for every vertex
	$v\in P$ the normal vectors to the facets meeting at $v$ can be chosen to form a basis of $\Zz^d$), and the function $\lambda$ is defined by the normal vectors of $P$, then $M(P,\mathit{\Lambda})$ is a projective toric manifold (cf. \cite{O88}).
	
	There is an equivalent way to define quasitoric manifolds from polytopal spheres.
	Let $\Delta$ be a polytopal $(d-1)$-sphere. Then the simple polyhedral complex $P_\Delta$ (see subsection \ref{subsec:m-a}) can be viewed as the dual simple polytope of the simplicial polytope which $\Delta$ bounds.
	Suppose $\lambda:\FF_0(\Delta)\to \Zz^d$, $i\mapsto\boldsymbol{\lambda}_i$ is a characteristic function, that is it satisfies the condition
	\begin{equation}\label{con:*}
		\mathrm{det}(\boldsymbol{\lambda}_{i_1},\dots,\boldsymbol{\lambda}_{i_d})=\pm1\quad \text{whenever }(i_1,\dots, i_d)\in\Delta.
	\end{equation}
	By means of $\lambda$, we get a $T^d$-space $M(\Delta,\mathit{\Lambda}):=P_\Delta\times T^d/\sim$ as in the construction \eqref{eq:m-a} of $\ZZ_\Delta$.
	Let $M(P_\Delta,\mathit{\Lambda})$ be the space constructed in the first way. It is obvious that $M(P_\Delta,\mathit{\Lambda})=M(\Delta,\mathit{\Lambda})$. For notational consistency, we use the second construction $M(\Delta,\mathit{\Lambda})$ to denote a quasitoric manifold in the rest of this paper.
	
	\begin{rem}
		(i) The condition \eqref{con:*} is equivalent to saying that the linear sequence $\{\theta_i=\lambda_{i1}x_1+\cdots+\lambda_{im}x_m\}_{1\leqslant i\leqslant d}$ is an integral l.s.o.p. for $\Zz[\Delta]$.
		
		(ii) For every $2$- or $3$-dimensional simple polytope $P$ there exists a quasitoric
		manifold over $P$. (The $3$-dimensional case is due to the Four Color Theorem.) But for $n\geqslant4$, there exist simple $n$-polytopes which do not arise as the base spaces of quasitoric manifolds, since the integral l.s.o.p. for a polytopal sphere $\Delta$ may fail to exist when $\dim\Delta\geqslant3$.
	\end{rem}
	
	Note that a characteristic matrix $\mathit{\Lambda}$ for $\Delta$  defines a map of lattices:
	$\mathit{\Lambda}:\Zz^m\to \Zz^d,\,\boldsymbol{e}_i\mapsto\boldsymbol{\lambda}_i$. Condition \eqref{con:*} implies that there is a short exact sequence
	\[0\to\Zz^{m-d}\to\Zz^m\xr{\mathit{\Lambda}}\Zz^d\to0.\]
	The matrix $\mathit{\Lambda}$ also induces an epimorphism of tori
	\[\exp\mathit{\Lambda}:T^m\to T^d,\quad T_i\mapsto\{(e^{2\pi i\lambda_{1i}t},\dots,e^{2\pi i\lambda_{di}t})\in T^d,\ t\in\Rr\},\]
	whose kernel we denote by $K_\mathit{\Lambda}$. Obviously $K_\mathit{\Lambda}=T^{m-d}$.
	From the construction of moment-angle manifolds and quasitoric manifolds we can easily see the following relation between them.
	\begin{prop}[{\cite[Proposition 7.3.12]{BP15}}]
		The group $K_\mathit{\Lambda}=T^{m-d}$ acts freely and smoothly on $\ZZ_\Delta$. There is a $T^d$-equaviriant homeomorphism
		\[\ZZ_\Delta/K_\mathit{\Lambda}\cong M(\Delta,\mathit{\Lambda}).\]
	\end{prop}
	The cohomology ring of a quasitoric manifold has a simple expression as follows.
	\begin{thm}[Davis-Januszkiewicz, \cite{DJ91}]\label{thm:quasi}
		Let $M(\Delta,\mathit{\Lambda})$ be a quasitoric manifold constructed from a polytopal $(d-1)$-sphere $\Delta$ with $m$ vertices,
		$\mathit{\Lambda}=(\lambda_{ij})$ be the corresponding characteristic $d\times m$ matrix.
		Then the cohomology ring $H^*(M(\Delta,\mathit{\Lambda});\mathbb{Z})$ is generated by the degree-two classes, and is given by
		\[H^*(M(\Delta,\mathit{\Lambda});\mathbb{Z})=\mathbb{Z}[\Delta]/\Theta,\]
		where $\Theta$ is the ideal generated by the linear forms $\lambda_{i1}x_1+\cdots+\lambda_{im}x_m,\ 1\leqslant i\leqslant m$.
	\end{thm}

	\subsection{Topological toric orbifolds and rational toric manifolds}\label{subsec:rational toric manifold}
	In fact, the construction of quasitoric manifolds over simple polytopes can be generalized to cases where the base space is an arbitrary simple polyhedral complex (cf. \cite[\S 2]{DJ91}). Now we discuss such a generalization, and introduce the central topological spaces of this paper.
	
	Let $\Delta$ be a simplicial complex of dimension $d-1$, $P_\Delta$ the simple polyhedral complex associated to $\Delta$.
	A map $\lambda:\FF_0(\Delta)\to \Zz^d,\ i\mapsto\boldsymbol{\lambda_i}=(\lambda_{1i},\dots,\lambda_{di})^T$ (here we require $\boldsymbol{\lambda_i}$ to be primitive in $\Zz^d$) is called a \emph{generalized characteristic function} if the linear sequence $\{\theta_i=\lambda_{i1}x_1+\cdots+\lambda_{im}x_m\}_{1\leqslant i\leqslant d}$ is an l.s.o.p. for $\Qq[\Delta]$.
	In this case, the $d\times m$ matrix $\mathit{\Lambda}=(\lambda_{ij})$ is called a \emph{generalized characteristic matrix}.
	Note that the rational face ring $\Qq[\Delta]$  always admits an l.s.o.p.
	
	For a \emph{generalized characteristic pair} $(\Delta,\mathit{\Lambda})$  consisting of a simplicial complex $\Delta$ and its generalized characteristic matrix $\mathit{\Lambda}$.
	We put $M(\Delta,\mathit{\Lambda})=P_\Delta\times T^d/\sim$, where the equivalence relation is defined exactly as in the case of quasitoric manifold;
	as before, $M(\Delta,\mathit{\Lambda})$ is a $T^d$-space over $P_\Delta$. We call $M(\Delta,\mathit{\Lambda})$ a \emph{toric space by D-J construction}.
	
	\emph{Terminology Convention.} To simplify terminologies, we will omit the word `generalized' and the words `by D-J construction' in the definitions above, since we always discuss such toric spaces in the rest of this paper.
	
	As we have seen, the matrix $\mathit{\Lambda}$  defines a map of lattices:
	$\mathit{\Lambda}:\Zz^m\to \Zz^d$, which can be extended to a exact sequence
	\[0\to\Zz^{m-d}\to\Zz^m\xr{\mathit{\Lambda}}\Zz^d\to G\to0.\]
	
	Since $\{\theta_1,\dots,\theta_m\}$ is an l.s.o.p for $\Qq[\Delta]$, $G$ is a finite group. Set $N=\mathrm{Im}\,\mathit{\Lambda}$.
	Then the  exact sequence above splits into two short exact sequences:
	\begin{gather*}
		0\to\Zz^{m-d}\to\Zz^m\xr{\mathit{\Lambda}}N\to 0,\\
		0\to N\to\Zz^d\to G\to 0.
	\end{gather*}
	We apply the functor $\otimes_\Zz S^1$ ($S^1\subset \Cc$) to them and we get
	\begin{gather*}
		0\to T^{m-d}\to T^m\to T^d_N\to 0,\\
		0\to \mathrm{Tor}_{\Zz}^1(G,S^1)=G\to T^d_N\to T^d\to 0=G\otimes_\Zz S^1.
	\end{gather*}
	Thus the lattice map $\mathit{\Lambda}$ induces an epimorphism of tori $\exp\mathit{\Lambda}:T^m\to T^d$ with kernel $K_\mathit{\Lambda}=T^{m-d}\times G$.
	
	Recall that an action of a group on a topological space is \emph{almost free} if all isotropy subgroups are finite.
	As in the case of quasitoric manifold, it can be shown that
	\begin{prop}[cf. {\cite[Theorem 4.8.5]{BP15}} and {\cite[Theorem 5.1.11]{CLS11}}]\label{prop:orbi}
		The group $K_\mathit{\Lambda}=T^{m-d}\times G$ acts almost freely and properly on $\ZZ_\Delta$. There is a $T^d$-equaviriant homeomorphism
		\[\ZZ_\Delta/K_\mathit{\Lambda}\cong M(\Delta,\mathit{\Lambda}).\]
	\end{prop}
	In particular, if $\Delta$ is a homology $(d-1)$-sphere, then $\ZZ_\Delta$ is a closed, orientable manifold of dimension $m+d$.
	So Proposition \ref{prop:orbi} implies that for a given characteristic matrix $\mathit{\Lambda}$, $M(\Delta,\mathit{\Lambda})$ is a closed, orientable $2d$-dimensional orbifold with a $T^d$-action.
	The orientation of $M(\Delta,\mathit{\Lambda})$ can be defined as follows.
	First, give an orientation to $X=\ZZ_\Delta/T^{m-d}$ by choosing orientations of $\ZZ_\Delta$ and $T^{m-d}$ respectively. Then since the $G$-action on $X$ extends to a toral action, it preserves the orientation. Thus $M(\Delta,\mathit{\Lambda})=X/G$ inherits an orientation from $X$. In this case, we call $M(\Delta,\mathit{\Lambda})$ a \emph{topological toric orbifold}. 
	
	Similarly, if $\Delta$ is a rational homology $(d-1)$-sphere, $\ZZ_\Delta$ is a closed, orientable, rational homology $(m+d)$-manifold.
	It follows that $M(\Delta,\mathit{\Lambda})$ is a closed, orientable, rational homology $2d$-manifold, called a \emph{rational toric manifold}.
	Furthermore, if $\Delta$ is a rational homology ball, $M(\Delta,\mathit{\Lambda})$ is an orientable, rational homology $2d$-manifold with boundary.
	
	\begin{exmp}Let $\Delta=\partial\Delta^2$, the boundary complex of a $2$-simplex. Hence $\ZZ_\Delta=S^5$ (Example \ref{exmp:m-a} (i)).
		We define $M(\Delta,\mathit{\Lambda})$ in three cases with different $\mathit{\Lambda}$.
		
		(i) Take $\mathit{\Lambda}$ to be
		\[\boldsymbol{\lambda}_1=(1,0)^T,\ \boldsymbol{\lambda}_2=(0,1)^T,\ \boldsymbol{\lambda}_3=(-1,-1)^T.\]
		In this case, $M(\Delta,\mathit{\Lambda})$ is the quotient space of $S^5\subset\Cc^3$ under the diagonal $S^1$ action,
		so $M(\Delta,\mathit{\Lambda})=\Cc P^2$ is a toric manifold.
		
		(ii) Let $a,\,b\neq0$ are relatively prime positive integers. Define $\mathit{\Lambda}$ to be
		\[\boldsymbol{\lambda}_1=(1,0)^T,\ \boldsymbol{\lambda}_2=(0,1)^T,\ \boldsymbol{\lambda}_3=(-a,-b)^T.\]
		Then $M(\Delta,\mathit{\Lambda})$ is the quotient space of $S^5\subset\Cc^3$ under a twisted $S^1$ action:
		\[
		M(\Delta,\mathit{\Lambda})=\{(z_1,z_2,z_3)\in S^5\}/\sim,\quad (z_1,z_2,z_3)\sim(t^a z_1,t^b z_2,t z_3),\ t\in S^1.
		\]
		This space is the so-called \emph{weighted projective space} $\Cc P^2_{(a,b,1)}$ with weight $(a,b,1)$ (cf. \cite{Kaw73}).
		
		(iii) Take $\mathit{\Lambda}$ to be
		\[\boldsymbol{\lambda}_1=(1,-1)^T,\ \boldsymbol{\lambda}_2=(1,2)^T,\ \boldsymbol{\lambda}_3=(-2,-1)^T.\]
		In this case, the kernel of $\mathit{\Lambda}:\Zz^3\to\Zz^2$  is $\Zz\cdot(1,1,1)$, and the cokernel of $\mathit{\Lambda}$ is $\Zz_3$, which is generated by $\frac{1}{3}(\boldsymbol{\lambda}_2+2\boldsymbol{\lambda}_3)$. From Proposition \ref{prop:orbi} we know that $M(\Delta,\mathit{\Lambda})=S^5/(S^1\times\Zz_3)=\Cc P^2/\Zz_3$ with the following action of $\Zz_3$
		\[\varepsilon(z_1:z_2:z_3)=(z_1:\varepsilon z_2:\varepsilon^2z_3),\]
		where $\varepsilon$ is a primitive root of unity of degree $3$. Such a space is known as a \emph{fake weighted projective space} (cf. \cite{Buc08}).
	\end{exmp}
	\begin{rem}\label{rem:stellar sub}
		Let $\mathit{\Lambda}=(\boldsymbol{\lambda}_1,\dots,\boldsymbol{\lambda}_m)$ be a characteristic matrix for $\Delta$. Then we can defined a characteristic matrix $\mathit{\Lambda'}$ for the stellar subdivision $\Delta(\sigma)$ ($\sigma=\{1,\dots,k\}\in\Delta$, $k\geqslant 2$) as follows.
		For any $(a_1,\dots,a_k)\in\Qq^k$ with $a_i\neq0$ for $1\leqslant i\leqslant k$, define
		\[\mathit{\Lambda}'=(\mathit{\Lambda}\mid\boldsymbol{\lambda}_{v_\sigma}),\quad \boldsymbol{\lambda}_{v_\sigma}=a_1\boldsymbol{\lambda}_{1}+\cdots+a_k\boldsymbol{\lambda}_{k}.\]
		It is easy to check that $\mathit{\Lambda}'$ is a characteristic matrix for $\Delta(\sigma)$.
	\end{rem}
	
	\subsection{Cohomology of toric spaces associated to Cohen-Macaulay complexes}\label{subsec:cohom toric}
	In this subsection, we assume $\Delta$ is a Cohen-Macaulay complex over $\Qq$, i.e., for every l.s.o.p. $\{\theta_1,\dots,\theta_d\}$ for $\Qq[\Delta]$, the face ring $\Qq[\Delta]$ is free as a $\Qq[\theta_1,\dots,\theta_d]$-module.
	
	For a  characteristic matrix $\mathit{\Lambda}$ for $\Delta$, let $T^{m-d}\subset T^m$ be the subtorus corresponding to the kernel $\Zz^{m-d}$ of the lattice map $\mathit{\Lambda}:\Zz^m\to\Zz^d$. First, let us consider the rational $T^{m-d}$-equivariant cohomology of $\ZZ_\Delta$.
	
	\begin{prop}\label{prop:equi toric}
		The rational $T^{m-d}$-equivariant cohomology of $\ZZ_\Delta$ is given by
		\[H^*_{T^{m-d}}(\ZZ_\Delta;\Qq)\cong\Qq[\Delta]/\Theta,\]
		where $\Theta$ is the ideal generated by the l.s.o.p. $\{\theta_i=\lambda_{i1}x_1+\cdots+\lambda_{im}x_m\}_{1\leqslant i\leqslant d}$ corresponding to $\mathit{\Lambda}=(\lambda_{ij})$.
	\end{prop}
	
	Before giving the proof, let us recall some notions about equivariant topology.
	For an $i$-torus ($i>0$) $T^i$, let $T^i\to ET^i\to BT^i$ be the universal principal $T^i$ bundle.
	Note that the universal principal $S^1$-bundle is the infinite Hopf bundle $S^\infty\to\Cc P^\infty$.
	So the classifying space $BT^i$ of the $i$-torus $T^i$ is the product $(\Cc P^\infty)^i$ of $i$ copies of $\Cc P^\infty$, and the total
	space $ET^i$ over $BT^i$ can be identified with the $i$-fold product of the infinite-dimensional sphere $S^\infty$.
	Let $X$ be a $T^i$-space. Then the $T^i$-equivariant cohomology of $X$ is isomorphic to the ordinary cohomology of the \emph{Borel construction} $ET^i\times_{T^i}X$.
	Here \[
	ET^i\times_{T^i}X:=ET^i\times X/\sim,
	\]
	where $(eg,x)\sim(e,gx)$ for any $e\in ET^i,\ x\in X,\ g\in T^i$.
	
	In addition, we may assume that the lattice map $\mathit{\Lambda}:\Zz^m\to\Zz^d$ is onto, so that $\Zz^d$ corresponds to a subtorus $T^d\subset T^m$ and $T^m=T^{d}\times T^{m-1}$. Indeed if this is not the case, suppose $\mathrm{Im}\,\mathit{\Lambda}=N\subset\Zz^d$.
	(Remember that $G=\Zz^d/N$ is a finite group.)
	Choose a basis of the lattice $N$ and then we can get another characteristic matrix $\mathit{\Lambda}'$ written in this basis.
	It is easy to see that there is a $d\times d$ matrix $B\in GL(d,\Qq)$ such that $B\mathit{\Lambda}=\mathit{\Lambda}'$.
	So the rational ideal generated by the l.s.o.p. $\Theta'$ corresponding to $\mathit{\Lambda}'$ is the same as $\Theta$, and therefore $\Qq[\Delta]/\Theta=\Qq[\Delta]/\Theta'$.
	
	\begin{proof}[Proof of Propsition \ref{prop:equi toric}]
		The coefficient $\Qq$ will be implicit throughout the proof. Consider the principal $T^d$-bundle:
		\[ET^d\times(ET^{m-d}\times_{T^{m-d}}\ZZ_\Delta)\to ET^{m}\times_{T^m}\ZZ_\Delta.\]
		The the Serre spectral sequence of this fibration has $E_2$-term
		\[E_2^{p,q}=H^p(ET^{m}\times_{T^m}\ZZ_\Delta;H^q(T^d)).\]
		According to Theorem \ref{thm:equiv}, $E_2=\Qq[\Delta]\otimes\Lambda[v_1,\dots,v_d]$, where $\Lambda[v_1,\dots,v_d]$ ($\mathrm{deg}\,v_i=1$) is the exterior algebra over $\Qq$.
		
		We assert that the differential $d_2$ of the $E_2$-term sends $v_i$ to $\theta_i\in\Qq[\Delta]$.
		To see this, we consider the bundle map:
		\[
		\xymatrix{
			T^m \ar[r] \ar[d]^-{\exp\mathit{\Lambda}} & ET^d\times(ET^{m-d}\times\ZZ_\Delta) \ar[r] \ar[d] & ET^{m}\times_{T^m}\ZZ_\Delta \ar@{=}[d]\\
			T^d \ar[r] & ET^d\times(ET^{m-d}\times_{T^{m-d}}\ZZ_\Delta) \ar[r] & ET^{m}\times_{T^m}\ZZ_\Delta}
		\]
		Theorem \ref{thm:equiv} shows that the Serre spectral sequence of the upper fibration has $E_2$-term
		\[E_2=\Qq[\Delta]\otimes\Lambda[y_1,\dots,y_m],\quad \mathrm{deg}\,y_i=1.\]
		The homomorphism $(\textrm{exp\,}\mathit{\Lambda})^*:H^1(T^d)\to H^1(T^m)$ can be identified with the dual map of $\mathit{\Lambda}:\Zz^m\to \Zz^d$.
		Hence $(\textrm{exp\,}\mathit{\Lambda})^*(v_i)=\lambda_{i1}y_i+\cdots+\lambda_{im}y_m$.
		Since we have $d_2(y_i)=x_i$ in the $E_2$-term of the Serre spectral sequence of the upper fibration (cf. Appendix \ref{appdx:cell rep}), the assertion is true.
		
		The fact that $\Qq[\Delta]$ is a free $\Qq[\theta_1,\dots,\theta_d]$-module implies that the Serre spectral sequence of the lower fibration collapses at the $E_3$-term: $E_3=\Qq[\Delta]/\Theta$. Notice that $ET^d\times(ET^{m-d}\times_{T^{m-d}}\ZZ_\Delta)$ is homotopy equivalent to $ET^{m-d}\times_{T^{m-d}}\ZZ_\Delta$. Hence
		\[H^*_{T^{m-d}}(\ZZ_\Delta)= H^*(ET^{m-d}\times_{T^{m-d}}\ZZ_\Delta)\cong\Qq[\Delta]/\Theta.\]
	\end{proof}
	
	In the Serre fibration $ET^{m-d}\times\ZZ_\Delta\to ET^{m-d}\times_{T^{m-d}}\ZZ_\Delta$, the projection onto the second factor of $ET^{m-d}\times\ZZ_\Delta$ descends to a projection $ET^{m-d}\times_{T^{m-d}}\ZZ_\Delta\to \ZZ_\Delta/T^{m-d}$, compose this with
	the quotient map $\ZZ_\Delta/T^{m-d}\to \ZZ_\Delta/K_\mathit{\Lambda}=M(\Delta,\mathit{\Lambda})$ if necessary we get a map
	\[p:ET^{m-d}\times_{T^{m-d}}\ZZ_\Delta\to M(\Delta,\mathit{\Lambda}).\]
	\begin{thm}\label{thm:cohomology of orbifold}
		For any $(d-1)$-dimensional complex $\Delta$ (not necessarily being Cohen-Macaulay), we have the following ring isomorphism
		\[p^*:H^*(M(\Delta,\mathit{\Lambda});\Qq)\cong H^*_{T^{m-d}}(\ZZ_\Delta;\Qq),\]
		which is induced by the quotient map $p$ above.
	\end{thm}
	We include the proof of Theorem \ref{thm:cohomology of orbifold} in Appendix \ref{appdx:A1} for the reader's convenience.
	
	Although the integral cohomology of rational toric manifolds often has torsion, and the ring structure
	is subtle even in the simplest case of weighted projective spaces (see \cite{Kaw73}), their rational cohomology has the same simple form as quaitoric manifolds.
	\begin{cor}\label{cor:coho toric orbi}
		If $\Delta$ is Cohen-Macaulay, then we have a ring isomorphism \[H^*(M(\Delta,\mathit{\Lambda});\Qq)\cong \Qq[\Delta]/\Theta.\]
	\end{cor}
	\begin{rem}
	If $\Delta$ is the underlying complex of a (rational) complete simplicial fan $\Sigma\subset \Rr^d$ and $\mathit{\Lambda}$ is induced by the ray vectors of $\Sigma$, then $M(\Delta,\mathit{\Lambda})$ is just the toric variety corresponding to $\Sigma$. In this case,  the above cohomology formula was proved by Danilov \cite{Dan78} (see also \cite[\S 12.4]{CLS11}).		
	\end{rem}
	
	As we have seen, every characteristic matrix for $\Delta$ defines an l.s.o.p. for $\Qq[\Delta]$. Conversely, if $\Theta$ is an l.s.o.p. for $\Qq[\Delta]$, then the associated $d\times m$ matrix $A=(\boldsymbol{a}_1,\dots,\boldsymbol{a}_m)$ can be written as $A=(\frac{1}{p_1}\boldsymbol{\lambda}_1,\dots,\frac{1}{p_m}\boldsymbol{\lambda}_m)$ with $p_i\in\Zz$, such that $\mathit{\Lambda}=(\boldsymbol{\lambda}_1,\dots,\boldsymbol{\lambda}_m)$ is a  characteristic matrix for $\Delta$.
	Let $\Theta_\mathit{\Lambda}$ be the l.s.o.p. corresponds to $\mathit{\Lambda}$. Then it is easy to see that $\Qq[\Delta]/\Theta\to \Qq[\Delta]/\Theta_\mathit{\Lambda},\,x_i\mapsto p_ix_i$ is a ring isomorphism.
	So we will do not distinguish the ring $\Qq[\Delta]/\Theta$ from the cohomology of $M(\Delta,\mathit{\Lambda})$ for Cohen-Macaulay complex $\Delta$ because of Corollary \ref{cor:coho toric orbi}.

	\section{Topology of rational toric manifolds and its applications}\label{sec:duality}
	Throughout this section, $\Delta$ is a Cohen-Macaulay complex of dimension $d-1$.
	By $(\Delta,\mathit{\Lambda})$ and $\Theta$, we denote a characteristic pair and the corresponding l.s.o.p.
	So $M(\Delta,\mathit{\Lambda})$ is a $2d$-dimensional toric space.
	The simplified notation $M_\Delta$ for $M(\Delta,\mathit{\Lambda})$ will be also used whenever it creates no confusion.
	
	\subsection{Local topology of toric spaces}\label{subsec:local}
	For $(\Delta,\mathit{\Lambda})$ and a subset $\Ss=\{i_1,\dots,i_j\}\subset[m]=\FF_0(\Delta)$, let $\mathit{\Lambda}_\Ss=(\boldsymbol{\lambda}_{i_1},\dots,\boldsymbol{\lambda}_{i_j})$ be the restricted $d\times j$ matrix, and let $\Theta_\Ss=r_\Ss(\Theta)$ be the image of $\Theta$ under the projection $r_\Ss:\Qq[\Delta]\to\Qq[\Delta_\Ss]$.
	For a $(k-1)$-face $\sigma=\{i_1,\dots,i_k\}\in\Delta$, setting $\Ss_\sigma=\FF_0(\mathrm{st}_\sigma\Delta)=\{i_1,\dots,i_j\}$, then we have a map of tori
	$\exp\mathit{\Lambda}_{\Ss_\sigma}:T^{j}\to T^d$ and get a $T^d$-space
	\[M_\sigma:=M(\mathrm{st}_\sigma\Delta,\mathit{\Lambda}_{\Ss_\sigma})=\ZZ_{\mathrm{st}_\sigma\Delta}/K_{\mathit{\Lambda}_{\Ss_\sigma}},\quad K_{\mathit{\Lambda}_{\Ss_\sigma}}:=\mathrm{Ker}\,\exp\mathit{\Lambda}_{\Ss_\sigma}.\]
	Since $\mathrm{st}_\sigma\Delta$ is clearly Cohen-Macaulay, $H^*(M_\sigma;\Qq)=\Qq[\mathrm{st}_\sigma\Delta]/\Theta_{\Ss_\sigma}$.
	
	On the other hand, let $T^{m-j}_{[m]\setminus\Ss_\sigma}$ be the coordinate subtorus corresponding to the subset $[m]\setminus\Ss_\sigma$.
	We get another $T^d$-space
	\[\hat M_\sigma:=(\ZZ_{\mathrm{st}_\sigma\Delta}\times T^{m-j}_{[m]\setminus\Ss_\sigma})/K_{\mathit{\Lambda}},\quad K_{\mathit{\Lambda}}:=\mathrm{Ker}\,\exp\mathit{\Lambda}.\]
	It is easy to see that $\hat M_\sigma$ is the quotient space of $M_\sigma$ under a finite group $G\subset T^d$ action: $G\times M_\sigma\to M_\sigma$, so their rational cohomology rings are isomorphic (cf. Appendix \ref{appdx:A1}). Let $\Phi_\sigma: \Qq[\Delta_\Ss]/\Theta_{\Ss_\sigma}\to \Qq[\mathrm{st}_\sigma\Delta]/\Theta_{\Ss_\sigma}$ be the natural map.
	Then, the composition map $\Psi_\sigma=\Phi_\sigma\circ r_{\Ss_\sigma}$ is induced by an inclusion $\psi_\sigma:\hat M_\sigma\hookrightarrow M_\Delta$, as shown in the following commutative diagram:
	\begin{equation}\label{eq:restriction}
		\begin{gathered}
			\xymatrix
			{\Qq[\Delta]/\Theta\ar[r]^-{\Psi_\sigma} \ar@{=}[d] & \Qq[\mathrm{st}_\sigma\Delta]/\Theta_{\Ss_\sigma} \ar@{=}[d]\\
				H^*(M_\Delta;\Qq) \ar[r]^-{\psi_\sigma^*} & H^*(\hat M_\sigma;\Qq)}
		\end{gathered}
	\end{equation}
	
	Now  consider the subcomplex $\mathrm{lk}_\sigma\Delta$.
	Reordering the vertices if necessary, there exists a matrix $A\in GL(d,\Zz)$  such that
	\begin{equation}\label{eq:matrix}
		A\cdot\mathit{\Lambda}_\sigma=
		\begin{pmatrix}
			U_\sigma\\0
		\end{pmatrix}
		\ \text{ and }\ A\cdot\mathit{\Lambda}_{\Ss_\sigma}=
		\left(\begin{array}{c|c}
			U_\sigma&B\\\hline
			0&\Gamma\end{array}\right),
	\end{equation}
	where $U_\sigma$ is a full rank $k\times k$ upper triangle matrix.
	It is easily verified that the $(d-k)\times(j-k)$ matrix $\Gamma$ is a characteristic matrix for $\mathrm{lk}_\sigma\Delta$.
	Thus, we can define a $(2d-2k)$-dimensional toric space $N_\sigma$ as
	\[N_\sigma:=M(\mathrm{lk}_\sigma\Delta,\Gamma)=\ZZ_{\mathrm{lk}_\sigma\Delta}/K_{\Gamma},\quad K_{\Gamma}:=\mathrm{Ker}\,\exp\Gamma.\]
	
	Viewing $\ZZ_{\mathrm{lk}_\sigma\Delta}$ as a subspace of $\ZZ_{\mathrm{st}_\sigma\Delta}$:
	\[\ZZ_{\mathrm{lk}_\sigma\Delta}=\{(x_1,\dots,x_j)\in (D^2)^j:x_i=0 \text{ for }i\in\sigma\}\subset\ZZ_{\mathrm{st}_\sigma\Delta}.\]
	Then $N_\sigma$ is a deformation retract of $M_\sigma$ induced by the deformation retraction from $\ZZ_{\mathrm{st}_\sigma\Delta}=(D^2)^k_\sigma\times\ZZ_{\mathrm{lk}_\sigma\Delta}$ onto $\ZZ_{\mathrm{lk}_\sigma\Delta}$, and $\pi_\sigma:M_\sigma\to N_\sigma$ is (rationally) an orientable $D^{2k}$-bundle.
	
	In particular, if $\Delta$ is a rational homology $(d-1)$-sphere (resp. rational homology $(d-1)$-ball), $\mathrm{lk}_\sigma\Delta$ is a rational homology $(d-k-1)$-sphere (resp. rational homology $(d-k-1)$-sphere or $(d-k-1)$-ball). So in this case, $N_\sigma$ is a rational toric
	$(2d-2k)$-manifold ( resp. rational toric
	$(2d-2k)$-manifold with or without boundary). Let us look at an example.
	\begin{exmp}\label{exmp:toric manifold}
		Let $\Delta$ be the boundary of a square with $\{1,3\}$ and $\{2,4\}$ as missing faces. So $\ZZ_\Delta=S^3\times S^3$ (see Example \ref{exmp:m-a}).
		Define $\mathit{\Lambda}$ to be
		\[\boldsymbol{\lambda}_1=(1,0)^T,\ \boldsymbol{\lambda}_2=(0,1)^T,\ \boldsymbol{\lambda}_3=(-1,-1)^T,\ \boldsymbol{\lambda}_4=(0,-1)^T.\]
		Then $M_\Delta$ is the connected sum $\Cc P^2\#\overline{\Cc P^2}$ (see \cite{OF70}), where $\overline{\Cc P^2}$ is the projective space with the reversed orientation. The kernel subtorus $K_\mathit{\Lambda}=T^2$ corresponds to the sublattice
		\[\Zz\cdot(1,1,1,0)\oplus\Zz\cdot(0,1,0,1).\]
		It is not hard to verify that $N_i=S^2$ for all $1\leqslant i\leqslant 4$;
		$M_i=\hat M_i=D^2\times S^2$ for $i=1,3$; and for $i=2,4$, $M_i=\hat M_i$ is the total space of a $D^2$-bundle over $S^2$ such that the boundary of $M_i$ is the Hopf bundle:
		\[S^1\to S^3\to S^2.\]
	\end{exmp}
	
	In the previous notations, we have a composition map
	\begin{equation}\label{eq:map}
		\rho_\sigma:N_\sigma\xr{\phi_\sigma}M_\sigma\xr{q_\sigma}\hat M_\sigma\xr{\psi_\sigma} M_\Delta,
	\end{equation}
	where $q_\sigma$ is the quotient map; $\phi_\sigma$ and $\psi_\sigma$ are inclusions, and there are induced ring isomorphisms
	\begin{equation}\label{eq:link=star}
		H^*(\hat M_\sigma)\xr[\cong]{q_\sigma^*}H^*(M_\sigma)\xr[\cong]{\phi_\sigma^*}H^*(N_\sigma).
	\end{equation}
	
	\subsection{Excision for rational toric manifolds with boundary}
	In this subsection, we assume $\Delta$ is a rational homology ball with characteristic matrix $\mathit{\Lambda}$, so that $M_\Delta$ is a rational toric manifolds with boundary. The following lemma can be used to calculate the relative cohomology of the pair $(M_\Delta,\partial M_\Delta)$.
	\begin{lem}[Excision]\label{lem:excision}
		Suppose a characteristic pair $(\Delta',\mathit{\Lambda}')$ satisfies that $\Delta'$ is a rational homology sphere of the same dimension as $\Delta$, $\Delta\subset\Delta'$ and $\mathit{\Lambda}$ is the restriction of $\mathit{\Lambda}'$ to $\Delta$. Let $D$ be the closure of $\Delta'-\Delta$, $\DD=\FF_0(D)$, and $r_\DD:\Qq[\Delta']/\Theta'\to\Qq[D]/\Theta'_{\DD}$.  Then we have an isomorphism
		\[H^*(M_\Delta,\partial M_\Delta;\Qq)\cong\mathrm{Ker}\,r_\DD.\]
	\end{lem}
	\begin{proof}
		Note that $D$ is a rational homology ball. Let $\Ss=\FF_0(\Delta')-\FF_0(\Delta)$, $i=|\Ss|$ and $\UU=\FF_0(\Delta')-\FF_0(D)$, $j=|\UU|$.  We can define spaces
		\[\begin{split}
			\hat M_\Delta=(\ZZ_\Delta\times T^i_\Ss)/K_{\mathit{\Lambda}'}\ &\text{ and }\ \hat M_D=(\ZZ_D\times T^j_\UU)/K_{\mathit{\Lambda}'},\\
			\text{where }K_{\mathit{\Lambda}'}:&=\mathrm{Ker}\,\exp\mathit{\Lambda}'.
		\end{split}\]
		As we showed in \S \ref{subsec:local} that $M_\sigma$ and $\hat M_\sigma$ are rational cohomology equivalent, in the same way, we have $H^*(M_\Delta,\partial M_\Delta;\Qq)\cong H^*(\hat M_\Delta,\partial\hat M_\Delta;\Qq)$ and $H^*(M_D;\Qq)\cong H^*(\hat M_D;\Qq)$. Using Corollary \ref{cor:coho toric orbi} and five-lemma, we can readily deduce that $H^*(M_{\Delta'},\hat M_D;\Qq)\cong \mathrm{Ker}\,r_\DD$. Since $M_{\Delta'}=\hat M_\Delta\cup\hat M_D$ and $\hat M_\Delta\cap\hat M_D=\partial\hat M_\Delta$,  $H^*(M_{\Delta'},\hat M_D;\Qq)\cong H^*(\hat M_\Delta,\partial\hat M_\Delta;\Qq)$ by excision.  So The lemma is proved.
	\end{proof}
	\begin{cor}\label{cor:relative coho}
		Let $I$ be the ideal of $\Qq[\Delta]$ generated by $\{\xx_\sigma:\sigma\in\Delta-\partial\Delta\}$, $\Theta$ an l.s.o.p. for $\Qq[\Delta]$ and $M_\Delta$ the corresponding rational toric manifold with boundary. Then we have an isomorphism
		\[H^*(M_\Delta,\partial M_\Delta;\Qq)\cong I/I\Theta.\]
	\end{cor}
	\begin{proof}
		Let $\Delta'=\Delta\cup_{\partial\Delta}\CC\Delta$. Then $\Delta'$ is a rational homology sphere, and there is an l.s.o.p. $\Theta'$ for $\Qq[\Delta']$ such that $\Theta'$, restricted to $\Delta$, is $\Theta$.
		Then we have a short exact sequence
		\[0\to I/(I\cap\Theta')\to\Qq[\Delta']/\Theta'\to\Qq[\Delta]/\Theta\to 0.\]
		According to Lemma \ref{lem:excision}, $H^*(M_\Delta,\partial M_\Delta;\Qq)\cong I/(I\cap\Theta')$. So it remains to prove that $I/(I\cap\Theta')=I/I\Theta$.
		
		Since $M_\Delta$ is an orientable rational homology manifold with boundary, from Lefschetz duality we have
		\[I/(I\cap\Theta')=H^*(M_\Delta,\partial M_\Delta)\cong H^*(M_\Delta)=\Qq[\Delta]/\Theta.\]
		By \cite[II. Theorem 7.3]{S96} $I$ is isomorphic to the canonical module of $\Qq[\Delta]$. So $\dim(I/I\Theta)=\dim(\Qq[\Delta]/\Theta)$ (see \cite[Theorem 3.3.5 (a)]{BH98} and \cite[I. Theorem 12.5 (b)]{S96}). Note that $I\Theta=I\Theta'\subset I\cap\Theta'$. Therefore, the natural surjection $I/I\Theta\to I/(I\cap\Theta')$ is actually an isomorphism.
	\end{proof}
	
	\subsection{Poincar\'e  duality of rational toric manifolds}\label{subsec:duality}
	In this subsection, we assume that $\Delta$ is a rational homology sphere (or ball) of dimension $d-1$.
	So $M_\Delta$ is a rational toric $2d$-manifold (resp. rational toric $2d$-manifold with boundary), and therefore it should have (rational) Poincar\'e duality (resp. Lefschetz duality) property.
	After choosing an orientation of $M_\Delta$, denote by $[M_\Delta]\in H_{2d}(M_\Delta;\Qq)$ (resp. $H_{2d}(M_\Delta,\partial M_\Delta;\Qq)$) the fundamental class of $M_\Delta$ (resp. $(M_\Delta,\partial M_\Delta)$).
	The following lemma plays an important role in this paper.
	\begin{lem}[Poincar\'e duality]\label{lem:duality}
		If $\Delta$ is a rational homology $(d-1)$-sphere, then the map defined by
		\[H^{2j}(M_\Delta;\Qq)\xr{[M_\Delta]\sfr}H_{2d-2j}(M_\Delta;\Qq)\] is an isomorphism for all $j$.
		Namely, the rational algebra $\Qq[\Delta]/\Theta$ is a Poincar\'e duality algebra.
		Moreover, for any $(k-1)$-face $\sigma=\{i_1,\dots, i_k\}\in\Delta$, we have \[\Qq\cdot[M_\Delta]\sfr\xx_\sigma=\Qq\cdot(\rho_\sigma)_*([N_\sigma]),\]
		where $[N_\sigma]\in H_{2d-2k}(N_\sigma;\Qq)$ is a rational fundamental class of $N_\sigma$, and $\rho_\sigma$ is defined by \eqref{eq:map}.
	\end{lem}
	\begin{proof}
		Since $M_\Delta$ is a rational homology manifold when $\Delta$ is a rational homology sphere, the first statement is obvious.
		For the second statement, let $e_\sigma=e^2_{i_1}\times\cdots\times e^2_{i_k}$ be the $2k$-cell of $\ZZ_\Delta$ defined in Remark \ref{rem:cell}.
		Define the `orbit cell' $\tilde e_\sigma$ to be the image of $e_\sigma$ in the orbit space $M_\Delta$ under the quotient map $\ZZ_\Delta\to M_\Delta$. (Remark: Actually, $\tilde e_\sigma$ may not be a cell in general, but a rational homology ball, which is homeomorphic to the quotient of $D^{2k}$ under a finite group $G\subset T^k$ action. However, a rational homology ball plays the same role as a cell in rational homology calculations. This is what we need.)
		
		We claim that $\xx_\sigma\in\Qq[\Delta]/\Theta$ is represented by the cocycle
		$\tilde e_\sigma^*\in C^{2k}(M_\Delta;\Qq)$ up to multiplication by an
		integer (see Appendix \ref{appdx:cell rep} for a proof).
		Similarly, the fundamental class $[M_\Delta]$ is represented by the cycle
		\[\pm\sum_{\tau\in\FF_{d-1}(\Delta)}\tilde e_\tau\in C_{2d}(M_\Delta;\Qq)\]
		(up to multiplication by an integer for each term). Hence, the cap product $[M_\Delta]\sfr\xx_\sigma$ is represented by
		\[\pm\sum_{\tau\in\FF_{d-k-1}(\mathrm{lk}_\sigma\Delta)}\tilde e_\tau, \quad \text{summing over all facets of }\mathrm{lk}_\sigma\Delta.\]
		(Compare with the relation between the cap product for $\ZZ_\Delta$ and the excision product for full subcomplexes of $\Delta$.) But this is just a representative of a fundamental class of $N_\sigma$.
	\end{proof}
	\begin{rem}
		The Poincar\'e duality of $\Qq[\Delta]/\Theta$ can also be obtained in a purely algebraic way \cite[I.12]{S96}. Lemma \ref{lem:duality} provides a topological explanation of this algebraic phenomenon.
	\end{rem}
	Similar to the Poincar\'e duality of rational toric manifolds, for rational toric manifolds with boundary we have
	\begin{lem}[Lefschetz duality]\label{lem:Lefschetz duality}
		If $\Delta$ is a rational homology $(d-1)$-ball, then the maps defined by
		\[
		\begin{split}
			H^{2j}(M_\Delta,\partial M_\Delta;\Qq)&\xr{[M_\Delta]\sfr}H_{2d-2j}(M_\Delta;\Qq),\ \text{ and }\\
			H^{2j}(M_\Delta;\Qq)&\xr{[M_\Delta]\sfr}H_{2d-2j}(M_\Delta,\partial M_\Delta;\Qq)
		\end{split}
		\]
		are isomorphisms for all $j$. Moreover, for $\xx_\sigma\in I/I\Theta$ with $I=(\xx_\sigma:\sigma\in \Delta-\partial\Delta)$ (resp. $\xx_\sigma\in\Qq[\Delta]$ ),
		$\Qq\cdot[M_\Delta]\sfr\xx_\sigma=\Qq\cdot(\rho_\sigma)_*([N_\sigma])$, where $[N_\sigma]\in H_{2d-2k}(N_\sigma;\Qq)$ is a rational fundamental class of $N_\sigma$ (resp. $(N_\sigma,\partial N_\sigma)$).
	\end{lem}
	
	\subsection{Applications of the Poincar\'e duality lemma}
	Restricting attention to closed rational toric manifolds for simplicity,  we assume $\Delta$ is a rational homology sphere throughout this subsection.
	As an application of Lemma \ref{lem:duality}, we have the following result which is an essential ingredient of this paper.
	\begin{prop}\label{prop:link}
		For a face $\sigma\in \Delta$, let $\rho_\sigma:N_\sigma\to M_\Delta$ be the map defined in \eqref{eq:map}.
		Then for every $1\leqslant i\leqslant d$, the map
		\[H^{2k}(M_\Delta;\mathbb{Q})\xr{\bigoplus\rho_\sigma^*}\bigoplus_{\sigma\in\FF_{i-1}(\Delta)} H^{2k}(N_\sigma;\Qq)\]
		is an injection for all $k\leqslant d-i$.
	\end{prop}
	\begin{proof}
		It is equivalent to prove that
		\begin{equation}\label{eq:surj}
			\bigoplus_{\sigma\in\FF_{i-1}(\Delta)} H_{2k}(N_\sigma;\Qq)\xr{\bigoplus(\rho_\sigma)_*} H_{2k}(M_\Delta;\Qq)
		\end{equation}
		is a surjection for $k\leqslant d-i$.
		
		First we will show that for any $1\leqslant i\leqslant d$, \eqref{eq:surj} holds for $k=d-i$. As a consequence of Lemma \ref{lem:duality}, we have the following commutative diagram:
		\begin{equation}\label{diagram:link}
			\begin{gathered}
				\xymatrix{
					\bigoplus\limits_{\sigma\in\FF_{i-1}(\Delta)}\Qq\cdot\xx_\sigma \ar[rr]^-{[M_\Delta]\sfr}_-\cong \ar@{->>}[dd] & & \bigoplus\limits_{\sigma\in\FF_{i-1}(\Delta)}H_{2d-2i}(N_\sigma;\Qq) \ar[dd]^-{\bigoplus(\rho_\sigma)_*}\\
					& &\\
					H^{2i}(M_\Delta;\Qq) \ar[rr]^-{[M_\Delta]\sfr}_-\cong & & H_{2d-2i}(M_\Delta;\Qq)}
			\end{gathered}
		\end{equation}
		Since the left vertical map is surjective (see \cite[Lemma III.2.4]{S96}), so is the right vertical map.
		
		Next for $k<d-i$, notice that for each face pair $\tau\supset\sigma$ with $\dim\sigma=i-1$, $\dim\tau=d-1-k$, there is a map $\rho_{\tau|\sigma}:N_\tau\to N_\sigma$, and $\rho_\tau$ factors through $N_\sigma$ by this map.
		Hence, we have the following commutative diagram:
		\begin{equation}\label{diagram:link surj}
			\begin{gathered}
				\xymatrix{
					\bigoplus\limits_{\sigma\in\FF_{i-1}(\Delta)}\big(\bigoplus\limits_{\substack{\tau\supset\sigma,\\\tau\in\FF_{d-1-k}(\Delta)}} H_{2k}(N_\tau;\Qq)\big) \ar[rr]^-{\bigoplus\bigoplus(\rho_{\tau|\sigma})_*} \ar@{->>}[dd]^-{\bigoplus\bigoplus(\rho_\tau)_*} & & \bigoplus\limits_{\sigma\in\FF_{i-1}(\Delta)}H_{2k}(N_\sigma;\Qq) \ar[dd]^-{\bigoplus(\rho_\sigma)_*} \\
					& \\
					H_{2k}(M_\Delta;\Qq) \ar@{=}[rr] & & H_{2k}(M_\Delta;\Qq)
				}
			\end{gathered}
		\end{equation}
		We have already seen that the left vertical map is surjective, so the right vertical map is surjective too. Thus, \eqref{eq:surj} holds for all $k\leqslant d-i$.
	\end{proof}
	\begin{rem}
		Proposition can be translated into a purely algebraic description, that is, for every $1\leqslant i\leqslant d-1$ and $k\leqslant d-i$, we have an injection
		\[(\Qq[\Delta]/\Theta)_{2k}\xr{\bigoplus \Psi_\sigma} \bigoplus_{\sigma\in\FF_{i-1}(\Delta)}(\Qq[\mathrm{st}_\sigma\Delta]/\Theta_{\Ss_\sigma})_{2k},\]
		where $\Ss_\sigma=\FF_0(\mathrm{st}_\sigma\Delta)$ and $\Psi_\sigma$ is as in \eqref{eq:restriction}. For the special case that $i=1$, this algebraic result is also obtained by Adiprasito \cite[Lemma 3.4]{A18}.
	\end{rem}
	From the proof of Proposition \ref{prop:link}, we can readily generalize it to
	\begin{thm}\label{thm:basis link}
		If $\{\xx_{\sigma_1},\dots,\xx_{\sigma_{h_k}}\}$ is a basis for $(\Qq[\Delta]/\Theta)_{2k}$, then the map
		\[(\Qq[\Delta]/\Theta)_{2j}\xr{\bigoplus \Psi_{\sigma_i}} \bigoplus_{1\leqslant i\leqslant h_k}(\Qq[\mathrm{st}_{\sigma_i}\Delta]/\Theta_{\Ss_{\sigma_i}})_{2j}.\]
		is an injection for $j<d-k$ and an isomorphism for $j=d-k$.
	\end{thm}
	\begin{proof}
		The isomorphism comes from diagram \eqref{diagram:link}, and the injection comes from  diagram \eqref{diagram:link surj} and the following lemma.
	\end{proof}
	
	\begin{lem}
		If $\{\xx_{\sigma_1},\dots,\xx_{\sigma_{h_k}}\}$ is a basis for $(\Qq[\Delta]/\Theta)_{2k}$, then for each $n>k$,  $(\Qq[\Delta]/\Theta)_{2n}$ is spanned by the face monomials   $\{\xx_{\tau}:\tau\in\bigcup_{i=1}^{h_k}\mathrm{st}_{\sigma_i}\Delta\}$.
	\end{lem}
	\begin{proof}
		Since $(\Qq[\Delta]/\Theta)_{2k+2}=(\Qq[\Delta]/\Theta)_2\cdot(\Qq[\Delta]/\Theta)_{2k}$, it is spanned by the monomials of the form $x_i\xx_{\sigma_j}$. If $i\not\in\sigma_j$ and $x_i\xx_{\sigma_j}\neq0$, then $\tau=\{i\}\cup\sigma_j\in\mathrm{st}_{\sigma_i}\Delta$, so assume $i\in\sigma_j$.
		Since $\theta_1,\dots,\theta_d$ is an l.s.o.p for $\Qq[\Delta]$, some linear combination of them has the form
		\[\gamma=x_i+\sum_{x_l\not\in\sigma_j}a_lx_l,\ a_l\in\Qq.\]
		Then in $\Qq[\Delta]/\Theta$ we have $x_i\xx_{\sigma_j}=(x_i-\gamma)\xx_{\sigma_j}$, and so reduce to the case $x_i\not\in\sigma_j$. This prove the case $n=k+1$.
		Doing this inductively for $n=k+2,\dots,d$, we get the conclusion of the lemma.
	\end{proof}
	
	Before proceeding further, let us define a combinatorial construction.
	\begin{Def}
		Let $\Delta$ be a pure simplicial complex of dimension $d-1$. For an integer $1\leq i\leq d$, the \emph{$i$th partial barycentric subdivision}  $\DD_i(\Delta)$ of $\Delta$ is recursively defined to be the simplicial complex obtained from $\DD_{i-1}(\Delta)$ (setting $\DD_0(\Delta)=\Delta$) by applying stellar subdivision operations at all faces $\sigma\in\FF_{d-i}(\Delta)$. (Note that $\FF_{d-i}(\Delta)\subset\FF_{d-i}(D_{i-1}(\Delta))$.)
	\end{Def}
	
	Roughly speaking, the $i$th partial barycentric subdivision arises when only the simplices of dimension $\geq d-i$ are barycentrically subdivided. 
	It is easily verified that $\DD_i(\Delta)$ is well defined, i.e., it dose not depend on the order of the stellar subdivision operations we perform. 
	
	Note that $\DD_{d-1}(\Delta)=\DD_d(\Delta)$ is the barycentric subdivision of $\Delta$, and for $i<d$
	\[\FF_0(\DD_i(\Delta))=\FF_0(\Delta)\cup\{v_\sigma:\sigma\in\Delta,\ \dim\sigma\geq d-i\}.\]
	Recall that $v_\sigma$ denotes the adding vertex in the stellar subdivision at the face $\sigma\in\Delta$.
	
	\begin{prop}\label{prop:link of derived}
		Let $\Delta$ be a rational homology $(d-1)$-sphere, and suppose $M_{\DD_i(\Delta)}$ is a rational toric manifold.
		For an integer $0<i<d$, let $\VV_i=\FF_0(\DD_i(\Delta))\setminus\FF_0(\Delta)$. Then for each $0\leqslant k<i$, we have an injection
		\[H^{2k}(M_{\DD_i(\Delta)};\Qq)\xr{\bigoplus\rho_{v_\sigma}^*}\bigoplus_{v_\sigma\in\VV_i} H^{2k}(N_{v_\sigma};\Qq).\]
		Here $N_{v_\sigma}$ is the rational toric manifold associated to $\mathrm{lk}_{v_\sigma}\DD_i(\Delta)$ defined in subsection \ref{subsec:local}.
	\end{prop}
	
	\begin{proof}
		As before, it is equivalent to show that
		\[
		\bigoplus_{v_\sigma\in\VV_i} H_{2k}(N_{v_\sigma};\Qq)\xr{\bigoplus(\rho_{v_\sigma})_*} H_{2k}(M_{\DD_i(\Delta)};\Qq)
		\]
		is surjective for $k<i$.
		
		As in the proof of Proposition \ref{prop:link}, we have a commutative diagram
		\[
		\xymatrix{
			\bigoplus\limits_{v_\sigma\in\VV_i}\big(\bigoplus\limits_{\substack{\tau\ni v_\sigma,\\\tau\in\FF_{d-1-k}(\DD_i(\Delta))}} H_{2k}(N_\tau;\Qq)\big) \ar[rr]^-{\bigoplus\bigoplus(\rho_{\tau|v_\sigma})_*} \ar@{->>}[dd]^-{\bigoplus\bigoplus(\rho_\tau)_*} & & \bigoplus\limits_{v_\sigma\in\VV_i} H_{2k}(N_{v_\sigma};\Qq) \ar[dd]^-{\bigoplus(\rho_{v_\sigma})_*}\\
			& &\\
			H_{2k}(M_{\DD_i(\Delta)};\Qq) \ar@{=}[rr] & & H_{2k}(M_{\DD_i(\Delta)};\Qq)
		}\]
		By Proposition \ref{prop:link},
		\[\bigoplus_{\tau\in\FF_{d-1-k}(\DD_i(\Delta))} H_{2k}(N_\tau;\Qq)\xr{\bigoplus(\rho_\tau)_*} H_{2k}(M_{\DD_i(\Delta)};\Qq)\]
		is a surjection. Note that if $k<i$, then any face $\tau\in\FF_{d-1-k}(\DD_i(\Delta))$ must contain at least one vertex $v_\sigma\in\VV_i$.
		It follows that the left vertical map is surjective in the diagram above, then so is the right vertical map.
	\end{proof}
	For a subset $\mathscr{A}\subset\FF_k(\Delta)$, define $\Ss(\Delta,\mathscr{A})$ to be the set of simplicial complexes obtained from
	$\Delta$ by a sequence of stellar subdivision operations at each face of $\mathscr{A}$.
	(In general, changing the order of stellar subdivision operations produces a different simplicial complex.)
	
	Let $\Delta$ be a rational homology $(d-1)$-sphere, $\Theta$ a generic l.s.o.p. for $\Qq[\Delta]$ and
	suppose $\{\xx_{\sigma_1},\dots,\xx_{\sigma_{h_k}}\}$ is a basis of $(\Qq[\Delta]/\Theta)_{2k}$.
	Set $\mathscr{A}_{k-1}=\{\sigma_1,\dots,\sigma_{h_k}\}\subset\FF_{k-1}(\Delta)$. Then for any $\Delta'\in\Ss(\Delta,\mathscr{A}_{k-1})$, we can give an l.s.o.p. $\Theta'$ for $\Qq[\Delta']$ such that it is $\Theta$ when restricted to the vertices of $\Delta$ (cf. Remark \ref{rem:stellar sub}).
	Let $\VV=\FF_0(\Delta')-\FF_0(\Delta)=\{v_{\sigma_1},\dots,v_{\sigma_{h_k}}\}$. Then $(\Qq[\Delta']/\Theta')_{2k}$ has a face monomial basis as follows.
	\begin{lem}\label{lem:basis}
		$(\Qq[\Delta']/\Theta')_{2k}$ has a basis of the form $\{\xx_{\tau_1},\dots,\xx_{\tau_s}\}$ such that
		$\tau_j\cap\VV\neq\varnothing$ for each $1\leqslant j\leqslant s$.
	\end{lem}
	\begin{proof}
		Without loss of generality, we may assume $\Delta=\Delta_0$, $\Delta'=\Delta_{h_k}$, and $\Delta_i=\Delta_{i-1}(\sigma_i)$ for $1\leqslant i\leqslant h_k$. For notational simplicity, we use $\Theta$ to denote the l.s.o.p. for all $\Qq[\Delta_i]$.
		Let $D_i$ ($i\geqslant1$) be the closure of $\Delta_i-\mathrm{st}_{v_{\sigma_i}}\Delta_i$. Then for each $1\leqslant i\leqslant h_k$, we have a short exact sequence
		\begin{equation}\label{eq:short exact sequence for subdivision}
			0\to J_i/(J_i\cap\Theta)\to\Qq[\Delta_i]/\Theta\to\Qq[D_i]/\Theta\to 0,
		\end{equation}
		where $J_i$ is the ideal generated by the vertex $v_{\sigma_i}$.
		
		From the short exact sequence
		\[0\to I/(I\cap\Theta)\to\Qq[\Delta]/\Theta\to\Qq[D_1]/\Theta\to 0,\quad I=(\xx_{\sigma_1}),\]
		it is easy to see that $\{\xx_{\sigma_2},\dots,\xx_{\sigma_{h_k}}\}$ is a basis of $(\Qq[D_1]/\Theta)_{2k}$. Thus the short exact sequence \eqref{eq:short exact sequence for subdivision} implies that $(\Qq[\Delta_1]/\Theta)_{2k}$ has a basis of the form
		$\{\xx_{\tau_1},\dots,\xx_{\tau_r}\}\cup\{\xx_{\sigma_2},\dots,\xx_{\sigma_{h_k}}\}$, where $v_{\sigma_1}\in\tau_j$ for all $1\leqslant j\leqslant r$.
		Doing this inductively for $i=2,3\dots$, we see that $(\Qq[\Delta']/\Theta)_{2k}$ has the desired basis.
	\end{proof}
	By using Lemma \ref{lem:basis} and the same argument as in the proof of Proposition \ref{prop:link}, we can get that:
	\begin{prop}\label{prop:injection of stellar sub}
		In the notation above, for any $\Delta'\in\Ss(\Delta,\mathscr{A}_{k-1})$ and $i\leqslant d-k$, we have an injection
		\[(\Qq[\Delta']/\Theta')_{2i}\to\bigoplus_{1\leqslant j\leqslant h_k}(\Qq[\mathrm{st}_{v_{\sigma_j}}\Delta']/\Theta')_{2i}.\]
	\end{prop}
	
	\section{Weak Lefschetz proporty and subdivisions of rational homolgoy spheres}\label{sec:stellar sub}
	Stellar subdivisions play an important role in piecewise-linear geometry.
	In this section, we investigate the problem that which stellar subdivisions of a rational homology sphere has WLP.
	(Babson-Nevo's paper \cite{BN10} is a good reference for the strong-Lefschetz property about this question.)
	We begin with a lemma which is needed later on.
	\begin{lem}\label{lem:join WLP}
		If a rational homology $(d-1)$-sphere $\Delta$ has the form $\Delta=\partial\Delta^n*\Delta'$ with $n\geqslant\lceil d/2\rceil$, then $\Delta$ has the WLP in the sense that for any l.s.o.p $\Theta$, there exists a WLE for $\Qq[\Delta]/\Theta$.
	\end{lem}
	\begin{proof}
		By Proposition \ref{prop:WLP}, we only need to show that there exists a linear form $\omega\in\Qq[\Delta]$ such that the map \[\cdot\omega:(\Qq[\Delta]/\Theta)_{2\lfloor d/2 \rfloor}\to(\Qq[\Delta]/\Theta)_{2\lfloor d/2 \rfloor+2}\] is a surjection.
		
		Choosing an arbitrary facet $\sigma\in\partial\Delta^n$, we have $\mathrm{st}_\sigma\Delta=\sigma*\Delta'$.
		Consider the short exact sequence
		\begin{equation}\label{eq:star}
			0\to I\to\Qq[\Delta]\to\Qq[\mathrm{st}_\sigma\Delta]\to 0.
		\end{equation}
		Suppose $\FF_0(\Delta^n)\setminus\sigma=\{i\}$, then it is easy to see that $I=(x_i)$.
		If we quotient out by $\Theta$ in \eqref{eq:star}, we obtain the short exact sequence
		\begin{equation}\label{eq:quotient star}
			0\to I/(I\cap\Theta)\to\Qq[\Delta]/\Theta\to\Qq[\mathrm{st}_\sigma\Delta]/\Theta_{\Ss_\sigma}\to 0,
		\end{equation}
		in which $\Ss_\sigma=\FF_0(\mathrm{st}_\sigma\Delta)$.
		
		Let $k=\dim\Delta'+1$. Then $k\leqslant\lfloor d/2\rfloor$, since $n\geqslant \lceil d/2\rceil$ and $n+k=d$ by assumption.
		As we have seen in subsection \ref{subsec:local}, $\Qq[\mathrm{st}_\sigma\Delta]/\Theta_{\Ss_\sigma}$ is isomorphic to the cohomology of the rational toric $2k$-manifold $N_\sigma$.
		Hence, we have $(\Qq[\mathrm{st}_\sigma\Delta]/\Theta_{\Ss_\sigma})_{2l}=0$ for $l\geqslant \lfloor d/2\rfloor+1$.
		It follows from \eqref{eq:quotient star} that
		\begin{equation}\label{eq:identity}
			(I/(I\cap\Theta))_{2l}=(\Qq[\Delta]/\Theta)_{2l},\quad \text{for }l\geqslant \lfloor d/2\rfloor+1.
		\end{equation}
		
		Let $R$ be the image of the map $\cdot x_i:\Qq[\Delta]/\Theta\to \Qq[\Delta]/\Theta$. Since $I=(x_i)$, it follows that $I/(I\cap\Theta)=R$.
		Combining this with \eqref{eq:identity} we get the surjection
		\[\cdot x_i:(\Qq[\Delta]/\Theta)_{2\lfloor d/2\rfloor}\to(\Qq[\Delta]/\Theta)_{2\lfloor d/2\rfloor+2}.\]
		The form $\omega=x_i$ is what we need.
	\end{proof}
	
	The WLP of a rational homology sphere and the one of its stellar subdivisions are related by the following algebraic result, which is proved initially by B\"ohm-Papadakis \cite{BS15}. Here we give a simpler proof.
	\begin{prop}\label{prop:stellar subdivision}
		Let $\Delta$ be a rational homology $(d-1)$-sphere. For the stellar subdivision $\Delta'=\Delta(\sigma)$ at a face $\sigma\in\FF_n(\Delta)$, we have:
		
		(a) If $\Delta$ has the WLP and $n\geqslant d/2$, then $\Delta'$ has the WLP.
		
		(b) If $\Delta'$ has the WLP and $n>d/2$, then $\Delta$ has the WLP.
	\end{prop}
	\begin{proof}
		(a) If $\Delta$ has the WLP, then for a pair $(\omega_0,\Theta)\in\WW(\Delta)$ we have a surjection
		\[
		\cdot\omega_0:(\Qq[\Delta]/\Theta)_{2\lfloor d/2\rfloor}\to(\Qq[\Delta]/\Theta)_{2\lfloor d/2\rfloor+2}.
		\]
		Since $\dim\sigma\geqslant d/2$, $\dim\mathrm{lk}_\sigma\Delta=d-\dim\sigma-2\leqslant\lfloor d/2\rfloor-2.$
		Therefore, as in the proof of Lemma \ref{lem:join WLP}, for the short exact sequence
		\[
		0\to I/(I\cap\Theta)\to\Qq[\Delta]/\Theta\to\Qq[\mathrm{st}_\sigma\Delta]/\Theta_{\Ss_\sigma}\to 0,
		\]
		we have
		\begin{equation}\label{eq:identity more}
			(I/(I\cap\Theta))_{2l}=(\Qq[\Delta]/\Theta)_{2l},\quad \text{for }l\geqslant \lfloor d/2\rfloor.
		\end{equation}
		So the map $(I/(I\cap\Theta))_{2\lfloor d/2\rfloor}\xr{\cdot\omega_0}(I/(I\cap\Theta))_{2\lfloor d/2\rfloor+2}$ is surjective too.
		
		On the other hand, let $\Theta'$ be a generic l.s.o.p. for $\Qq[\Delta']$. Here we may assume $\Theta$ is the restriction of $\Theta'$ to $x_1,\dots,x_m$.
		Note that $\mathrm{lk}_{v_\sigma}\Delta'=\partial\sigma*\mathrm{lk}_\sigma\Delta$ is a $(d-2)$-sphere, then by Lemma \ref{lem:join WLP} and equation \eqref{eq:link=star}, for any $i\in\sigma$, $x_i$ is a WLE for $\Qq[\mathrm{st}_{v_\sigma}\Delta']/\Theta'_{\Ss_{v_\sigma}}$, where $\Ss_{v_\sigma}=\FF_0(\mathrm{st}_{v_\sigma}\Delta')$.
		Let $\omega$ be a generic linear combination of $\omega_0$ and $x_i$. Now consider the following commutative diagram of exact sequences:
		\begin{equation}\label{eq:diagram}
			\begin{gathered}
				\xymatrix{
					0 \ar[r]& I'/(I'\cap\Theta')\ar[d]^{\cdot\omega}\ar[r]&\Qq[\Delta']/\Theta'\ar[d]^{\cdot\omega}\ar[r]&\Qq[\mathrm{st}_{v_\sigma}\Delta']/\Theta'_{\Ss_{v_\sigma}}\ar[d]^{\cdot\omega}\ar[r]&0\\
					0 \ar[r]& I'/(I'\cap\Theta')\ar[r]&\Qq[\Delta']/\Theta'\ar[r]&\Qq[\mathrm{st}_{v_\sigma}\Delta']/\Theta'_{\Ss_{v_\sigma}}\ar[r]&0}
			\end{gathered}
		\end{equation}
		(For notational convenience, we omit the degree subscripts in both rows, which are $2\lfloor d/2\rfloor$ and $2\lfloor d/2\rfloor+2$ resp.)
		Thus, the right vertical map is a surjection. By Lemma \ref{lem:excision}, we have $I'/(I'\cap\Theta')\cong I/(I\cap\Theta)$,
		so the left vertical map is a surjection too. The five-lemma then gives a surjection
		\[\cdot\omega:(\Qq[\Delta']/\Theta)_{2\lfloor d/2\rfloor}\to(\Qq[\Delta']/\Theta)_{2\lfloor d/2\rfloor+2}.\]
		Hence $\Delta'$ has the WLP.
		
		(b) We divide the proof into two cases, depending on the parity of $d$.
		Consider first the case that $d$ is odd. In this case, $n>d/2$ is equivalent to $n\geqslant\lfloor d/2\rfloor+1$. Let the degree of the first and second rows in \eqref{eq:diagram} are $2\lfloor d/2\rfloor$ and $2\lfloor d/2\rfloor+2$ resp. Thus if $\Delta'$ has WLP, then for a generic choice of $(\omega,\Theta)\in\WW(\Delta')$,
		the middle vertical map is an isomorphism by Poincar\'e duality, and so the left vertical map is an injection.
		Combining this with \eqref{eq:identity more} and $h_{\lfloor d/2\rfloor}(\Delta)=h_{\lfloor d/2\rfloor+1}(\Delta)$ we see that the map
		\[\cdot\omega:(\Qq[\Delta]/\Theta)_{2\lfloor d/2\rfloor}\to(\Qq[\Delta]/\Theta)_{2\lfloor d/2\rfloor+2}\]
		is an isomorphism.
		So $\Delta$ has the WLP.
		
		For the case that $d$ is even, we have $d/2=\lfloor d/2\rfloor$. So $n>d/2$ implies that $\dim\mathrm{lk}_\sigma\Delta\leqslant\lfloor d/2\rfloor-3$, and therefore by the same argument as before, we have
		\begin{equation}\label{eq:identity even more}
			(I/(I\cap\Theta))_{2l}=(\Qq[\Delta]/\Theta)_{2l},\quad \text{for }l\geqslant \lfloor d/2\rfloor-1.
		\end{equation}
		Note that when $d$ is even, the WLP of $\Delta$ is equivalent to the injectivity of
		\begin{equation}\label{eq:inj stellar}
			\cdot\omega:(\Qq[\Delta]/\Theta)_{2\lfloor d/2\rfloor-2}\to(\Qq[\Delta]/\Theta)_{2\lfloor d/2\rfloor}\quad \text{for some } \omega.
		\end{equation}
		This time let the degree of the first and second rows in \eqref{eq:diagram} are $2\lfloor d/2\rfloor-2$ and $2\lfloor d/2\rfloor$ resp.
		The WLP of $\Delta'$ implies that for a generic choice of $(\omega,\Theta)\in\WW(\Delta')$, the middle vertical map is an injection. So the left vertical map is also an injection. It follows from Lemma \ref{lem:excision} and \eqref{eq:identity even more} that the map in \eqref{eq:inj stellar} is an injection too. So $\Delta$ has the WLP.
	\end{proof}
	
	\begin{cor}\label{cor:WLP of link}
		Let $\Delta$ be a rational homology $(d-1)$-sphere, $\VV_i=\FF_0(\DD_i(\Delta))\setminus\FF_0(\Delta)$.
		If $i\leqslant\lceil d/2\rceil$, then for every $v_\sigma\in\VV_i$, $\mathrm{lk}_{v_\sigma}\DD_i(\Delta)$ has the WLP.
	\end{cor}
	\begin{proof}
		It is easy to see that $\FF_0(\Delta)=\VV_0\subset\VV_1\subset\cdots\subset\VV_i$.
		Let $k_\sigma=d-\dim\sigma$. Then $k_\sigma$ is the smallest number such that $v_\sigma\in\VV_{k_\sigma}$, and \[\mathrm{lk}_{v_\sigma}\DD_{k_\sigma}(\Delta)=\partial\sigma*\mathrm{lk}_\sigma\DD_{k_\sigma-1}(\Delta)\]
		is a rational homology $(d-2)$-sphere.
		So when $k_\sigma\leqslant i\leqslant\lceil d/2\rceil$, we have $\dim\sigma\geqslant\lfloor d/2\rfloor=\lceil (d-1)/2\rceil$, and then
		$\mathrm{lk}_{v_\sigma}\DD_{k_\sigma}(\Delta)$ has the WLP by Lemma \ref{lem:join WLP}.
		
		When $i>k_\sigma$, $\mathrm{lk}_{v_\sigma}\DD_{i}(\Delta)$ is obtained from $\mathrm{lk}_{v_\sigma}\DD_{k_\sigma}(\Delta)$ by a sequence of stellar subdivisions at some faces of dimension greater than or equal to $d-i\geqslant\lfloor d/2\rfloor\geqslant (d-1)/2$.
		Hence we conclude our assertion with the help of Propositon \ref{prop:stellar subdivision} (a).
	\end{proof}
	
	\begin{cor}\label{cor:equiv of WLP}
		Let $\Delta$ be a rational homology $(d-1)$-sphere. Then for $i< d/2$, $\DD_i(\Delta)$ has the WLP if and only if $\Delta$ has the WLP.
	\end{cor}
	\begin{proof}
		By definition, $\DD_i(\Delta)$ is obtained from $\Delta$ by a sequence of stellar subdivisions at faces of dimension greater than or equal to $d-i>d/2$.
		Then the conclusion follows from Proposition \ref{prop:stellar subdivision}.
	\end{proof}
	
	As an application of the previous results, we get the following theorem.
	\begin{thm}\label{thm:WLP of derived}
		For any rational homology $(d-1)$-sphere $\Delta$ and a generic l.s.o.p. $\Theta$ for $\Qq[\DD_k(\Delta)]$ with $k\leqslant\lceil d/2\rceil$, there exists a linear form $\omega\in\Qq[\DD_k(\Delta)]$ such that the map
		\[\cdot\omega:(\Qq[\DD_k(\Delta)]/\Theta)_{2i-2}\to (\Qq[\DD_k(\Delta)]/\Theta)_{2i}\]
		is an injection for $i\leqslant\min\{k,\lfloor d/2\rfloor\}$. In particular, if $d$ is even, and $k=d/2$, there is a WLE for $\Qq[\DD_k(\Delta)]/\Theta$.
	\end{thm}
	\begin{proof}
		Consider the morphism
		\[\Qq[\DD_{k}(\Delta)]/\Theta\xr{\bigoplus \Psi_{v_\sigma}}\bigoplus_{v_\sigma\in\VV_k}\Qq[\mathrm{st}_{v_\sigma}\DD_{k}(\Delta)]/\Theta_{\Ss_{v_\sigma}},\]
		where $\VV_k=\FF_0(\DD_k(\Delta))\setminus\FF_0(\Delta)$ and $\Ss_{v_\sigma}=\FF_0(\mathrm{st}_{v_\sigma}\DD_{k}(\Delta))$.
		Since $k\leqslant\lceil d/2\rceil$,
		it follows by Corollary \ref{cor:WLP of link} and Lemma \ref{lem:generic} that for a generic choice of linear form $\omega$, and for $i\leqslant\lfloor d/2\rfloor$, the right vertical map is injective in the following commutative diagram:
		\[\begin{CD}
			(\Qq[\DD_{k}(\Delta)]/\Theta)_{2i-2} @>\bigoplus \Psi_{v_\sigma}>>\bigoplus_{v_\sigma\in\VV_k}(\Qq[\mathrm{st}_{v_\sigma}\DD_{k}(\Delta)]/\Theta_{\Ss_{v_\sigma}})_{2i-2} \\
			@VV\cdot\omega V @VV\cdot\omega V\\
			(\Qq[\DD_{k}(\Delta)]/\Theta)_{2i} @>\bigoplus \Psi_{v_\sigma}>>\bigoplus_{v_\sigma\in\VV_k}(\Qq[\mathrm{st}_{v_\sigma}\DD_{k}(\Delta)]/\Theta_{\Ss_{v_\sigma}})_{2i}
		\end{CD}\]
		Proposition \ref{prop:link of derived} implies that the upper horizontal map is an injection for $i\leqslant k$, so the left vertical map is also an injection for $i\leqslant\min\{k,\lfloor d/2\rfloor\}$.
		When $d$ is even and $k=d/2$, this is an equivalent condition for $\omega$ to be a WLE.
	\end{proof}
	
	\begin{rem}
		For an odd-dimensional rational homology $(d-1)$-sphere $\Delta$, Corollary \ref{cor:equiv of WLP} says that if $\DD_{d/2-1}(\Delta)$ has the WLP then so dose $\Delta$.  However, it seems that we can only get the WLP of $\DD_{d/2}(\Delta)$ from Theorem \ref{thm:WLP of derived}.
		
		Compare the last statement of Theorem \ref{thm:WLP of derived} with the result in \cite{MY14}, which says that the barycentric subdivision of an odd-dimensional Cohen-Macaulay complex has the WLP.
	\end{rem}
	
	\begin{cor}
		Let $\Delta$ be a rational homology $(d-1)$-sphere. Then the $g$-conjecture holds for $\DD_{\lceil d/2\rceil}(\Delta)$ and $\DD_{\lfloor d/2\rfloor}(\Delta)$.
	\end{cor}
	
	\begin{thm}\label{thm:stellar WLP}
		In the notation of Proposition \ref{prop:injection of stellar sub}, if $k\geqslant \lceil \frac{d-1}{2}\rceil$, then for any $\Delta'\in\Ss(\Delta,\mathscr{A}_{k})$, there exists a linear form $\omega\in\Qq[\Delta']$ such that the map
		\[\cdot\omega:(\Qq[\Delta']/\Theta')_{2i-2}\to (\Qq[\Delta']/\Theta')_{2i}\]
		is an injection for $i\leqslant\min\{d-k,\lfloor d/2\rfloor\}$. Especially, if $d$ is even, and $k=d/2$, then there is a WLE for $\Qq[\Delta']/\Theta'$.
	\end{thm}
	\begin{proof}
		By Proposition \ref{prop:injection of stellar sub}, we have an injection
		\[(\Qq[\Delta']/\Theta')_{2i}\to\bigoplus_{1\leqslant j\leqslant h_k}(\Qq[\mathrm{st}_{v_{\sigma_j}}\Delta']/\Theta')_{2i}\quad \text{for }i\leqslant d-k-1.\]
		Since $\mathrm{lk}_{v_{\sigma_j}}\Delta'$ is a $(d-2)$-sphere, and has the form $\partial\sigma_j*L$,
		it follows that if $\dim\sigma_j=k\geqslant\lceil \frac{d-1}{2}\rceil$, there is a WLE for $\Qq[\mathrm{st}_{v_{\sigma_j}}\Delta']/\Theta'$ by Lemma \ref{lem:join WLP}.
		The same reasoning as in the proof of Theorem \ref{thm:WLP of derived} gives the desired result.
	\end{proof}
	\begin{cor}\label{cor:g-conj}
		Let $\Delta$ be a rational homology $(d-1)$-sphere. Then the $g$-conjecture holds for any $\Delta'\in \Ss(\Delta,\mathscr{A}_{\lceil d/2\rceil})$ or $\Delta'\in\Ss(\Delta,\mathscr{A}_{\lfloor d/2\rfloor})$.
	\end{cor}

	\section{Toric spaces associated to Buchsbaum complexes}\label{sec:buchsbaum complex}
	In this section, we consider toric spaces associated to rational Buchsbaum complexes. This class of simplicial complexes plays a significant role in algebraic combinatorics and includes rational homology manifolds as a special subclass.
	We start with some general results about the cohomology properties of toric spaces by the D-J construction.
	\subsection{A decomposition  of toric spaces}
	In this subsection, $\Delta$ is an arbitrary simplicial complex of dimension $d-1$, and $M_\Delta=M(\Delta,\mathit{\Lambda})$ is a toric space associated to $\Delta$.
	As we have seen in \ref{subsec:rational toric manifold}, $M_\Delta$ is the quotient space $\CC\Delta\times T^d/\sim $. So, $M_\Delta$ is the union of two spaces:
	\begin{equation}\label{eq:decompositon}
		M_\Delta=(\CC\Delta\times T^d)\cup_{\Delta\times T^d} ({I\times\Delta\times T^d}/\sim),
	\end{equation}
	where $I=[0,1]$, with the relation `$\sim$' defined on $\{0\}\times \Delta'\times T^d$ and the gluing identity map defined on $\{1\}\times \Delta\times T^d$.
	So we have a long exact sequence
	\begin{equation}\label{eq:long exact}
		\begin{split}
			\cdots\to H^*(M_\Delta,{I\times\Delta\times T^d}/\sim)\xr{j^*}&H^*(M_\Delta)\xr{i^*}H^*({I\times\Delta\times T^d}/\sim)\\
			&\xr{\partial} H^{*+1}(M_\Delta,{I\times\Delta\times T^d}/\sim)\to\cdots
		\end{split}
	\end{equation}
	By excision and the fact that $(\CC\Delta\times T^d)/(\Delta\times T^d)=\Sigma\Delta\wedge T^d_+$, we have
	\[H^*(M_\Delta,{I\times\Delta\times T^d}/\sim)\cong H^*({\CC\Delta\times T^d},{\Delta\times T^d})\cong \w H^*(\Sigma\Delta)\otimes H^*(T^d).\]
	
	\begin{lem}\label{lem:product vanishing}
		In the natation above, we have $\mathrm{Im}\,j^*\cdot\w H^*(M_\Delta)=0$
	\end{lem}
	\begin{proof}
		Consider another long exact sequence
		\[\begin{split}
			\cdots\to H^*(M_\Delta,\CC\Delta\times T^d)\xr{f^*}H^*(M_\Delta)&\xr{g^*}H^*(\CC\Delta\times T^d)\\
			&\xr{\partial} H^{*+1}(M_\Delta,\CC\Delta\times T^d)\to\cdots
		\end{split}\]
		Since $T^d$ is contractible in $M_\Delta$, $H^k(M_\Delta)\xr{g^*}H^k(\CC\Delta\times T^d)$ is a zero map for $k>0$. Exactness of the sequence then implies that $H^*(M_\Delta,\CC\Delta\times T^d)\xr{f^*}\w H^*(M_\Delta)$ is onto. On the other hand, from the definition of cup product, we can see that $j^*(\alpha)\ssm f^*(\beta)=0$ for any $\alpha\in H^*(M_\Delta,I\times\Delta\times T^d/\sim)$ and $\beta\in H^*(M_\Delta,\CC\Delta\times T^d)$. Thus, the lemma is proved.
	\end{proof}
	\begin{lem}\label{lem:zero mapping}
		For any nonempty face $\sigma\in \Delta$, the following composition is zero.
		\[H^*(M_\Delta,{I\times\Delta\times T^d}/\sim)\xr{j^*} H^*(M_\Delta)\to H^*(M_\sigma),\]
		where $M_\sigma\subset M_\Delta$ is the toric space defined in \S \ref{subsec:local}.
	\end{lem}
	\begin{proof}
		The composition in the lemma is the same as the composition
		\[H^*(M_\Delta,{I\times\Delta\times T^d}/\sim)\to H^*(M_\sigma,{I\times\mathrm{st}_\sigma\Delta\times T^d}/\sim)\to H^*(M_\sigma).\]
		The middle relative cohomology is isomorphic to the reduced cohomology of
		\[(\CC(\mathrm{st}_\sigma\Delta)\times T^d)/(\mathrm{st}_\sigma\Delta\times T^d)=\Sigma(\mathrm{st}_\sigma\Delta)\wedge T^d_+\simeq pt.\]
		The right homotopy equivalent follows from the fact that $\mathrm{st}_\sigma\Delta$ is contractible.
		Hence this composition factors through a zero term, and so itself is zero.
	\end{proof}
	
	\subsection{Cohomology of toric spaces associated to Buchsbaum complexes}\label{subsec:buchsbaum}
	In this subsection, we will give topological proofs of several fundamental algebraic results about Buchsbaum complexes.
	\begin{lem}\label{lem:cokernel}
		Suppose $\Delta$ is a Buchsbaum complex.
		Let $\Delta\times T^d=\{0\}\times \Delta\times T^d$, and $\pi:\Delta\times T^d\to\Delta\times T^d/\sim$ the quotient map in the definition of $M_\Delta$.
		Then the following composition is zero for any $q>p\geqslant0$.
		\[H^{p+q}({\Delta\times T^d}/\sim;\Qq)\xr{\pi^*}H^{p+q}({\Delta\times T^d};\Qq)\to H^p(\Delta;\Qq)\otimes H^q(T^d;\Qq).\]
	\end{lem}
	Since the proof of Lemma \ref{lem:cokernel} needs more complicated topological arguments, we put it in Appendix \ref{appdx:proof of cokernel lemma}.
	\begin{cor}\label{cor:cokernel}
		In the notations of the discussion preceding Lemma \ref{lem:product vanishing}, if $\Delta$ is Buchsbaum, then the restriction of $j^*$ to the cohomology subgroup
		\[\bigoplus_{q\geqslant p}\w H^p(\Sigma\Delta;\Qq)\otimes H^q(T^d;\Qq)\subset H^*(M_\Delta,{I\times\Delta\times T^d}/\sim;\Qq)\]
		is an injection.
	\end{cor}
	\begin{proof}
		The statement of Lemma \ref{lem:cokernel} is equivalent to saying that the cohomology subgroup $\bigoplus_{q\geqslant p}\w H^p(\Sigma\Delta;\Qq)\otimes H^q(T^d;\Qq)$ is not in the image of the boundary homomorphisms $\partial$ in the sequence \eqref{eq:long exact}. So exactness of the sequence gives the result.
	\end{proof}
	Recall the following fibration in the proof of Proposition \ref{prop:equi toric}
	\[T^d \to ET^m\times_{T^{m-d}}\ZZ_\Delta \to ET^{m}\times_{T^m}\ZZ_\Delta,\]
	where $T^d$ is the torus associated to an l.s.o.p. $\Theta=\{\theta_1,\dots,\theta_d\}$ for $\Qq[\Delta]$.
	For the Serre spectral sequence of this fibration, the $E_2$-term, as we have seen in the proof of Proposition \ref{prop:equi toric}, is
	\[
	E_2=\Qq[\Delta]\otimes\Lambda[v_1,\dots,v_d],\quad d_2(v_i)=\theta_i.
	\]
	For the $E_3$-term, we have the following result.
	\begin{lem}\label{lem:E3-page}
		If $\Delta$ is a Buchsbaum complex, then the $E_3$-term of the rational Serre spectral sequence of the  fibration above is $E_3^{p,q}=0$ if $p$ is odd, and
		\[
		\dim E_3^{2p,q}=
		\begin{cases}
			\tbinom{d}{p+q}\w\beta_{p-1}(\Delta)\quad&\text{for }q>0,\\
			\\
			h_p(\Delta)-\tbinom{d}{p}\sum_{i=1}^{p-1}(-1)^i\w \beta_{p-i-1}(\Delta)\quad&\text{for }q=0.
		\end{cases}\]
	\end{lem}
	We will prove this lemma by using a double complex, whose calculation is due to Adiprasito \cite[Lemma 3.4 and Proposition 3.9]{A18}.
	Note that the formula for $E_3^{2p,0}$ is just Schenzel's formula (Theorem \ref{thm:schenzel}).
	
	\begin{proof}[Proof of Lemma \ref{lem:E3-page}]
		For each $\sigma\in\Delta$, define a differential graded algebra
		\begin{gather*}
			\Ll^*_\sigma:=(\Qq[\mathrm{st}_\sigma\Delta]\otimes\Lambda[v_1,\dots,v_d],\, d),\quad dv_i=\theta_i,\quad dx_i=0;\\
			\mathrm{deg}\,v_i=-1,\quad \mathrm{deg}\,x_i=0.
		\end{gather*}
		Particularly, when $\sigma=\varnothing$, we briefly write $\Ll^*_\varnothing$ as $\Ll^*$. Note that $\Ll^*_\sigma$ as a $\Qq[m]$-module has another even internal grading which is preserved by the differential. We will denote this internal grading by subscript.
		
		Viewing $\{U_i:=\mathrm{st}_i\Delta\}_{i\in\FF_0(\Delta)}$ as a open cover of $\Delta$,
		for a subset $\sigma=\{i_1,\dots,i_k\}\subset[m]$, we formally set  the intersection $U_\sigma:=U_{i_1}\cap\cdots\cap U_{i_k}$ to be
		\[U_\sigma=
		\begin{cases}
			\varnothing\ &\text{if }\sigma\not\in\Delta,\\
			\mathrm{st}_\sigma\Delta\ &\text{if }\sigma\in\Delta\ \text{ (its usual form).}
		\end{cases}\]
		Now we define a double complex $(\RR,\delta,d)$ by
		\begin{gather*}
			\RR=\bigoplus_{p,q\geqslant 0}\RR^{p,-q},\quad
			\RR^{p,-q}=\bigoplus_{\sigma\in\FF_{p-1}(\Delta)}\Ll^{-q}_\sigma,\\
			\delta:\RR^{p,-q}\to \RR^{p+1,-q},\quad d:\RR^{p,-q}\to \RR^{p,-(q-1)}.
		\end{gather*}
		Here $\delta$ is the \v  Cech coboundary operator with respect to the intersection rules we formally set above. Its total complex is $\RR^*=\bigoplus_k\RR^k=\bigoplus_k\bigoplus_{p-q=k}\RR^{p,-q}$ with total differential $D=\delta+(-1)^pd$.
		There are two spectral sequences converging to the total cohomology $H^*(\RR^*,D)$.
		One spectral sequence starts with $\mathrm{^I}E_1=H_\delta$ and $\mathrm{^I}E_2=H_d H_\delta$, and another with $\mathrm{^{II}}E_1=H_d$ and $\mathrm{^{II}}E_2=H_\delta H_d$.
		
		By Hochster's theorem (unpublished, see \cite[Stanley Theorem II.4.1]{S80}), we have
		\[\mathrm{^I}E_1^{p,-q}=\w H^{p-1}(\Delta;\Qq)\otimes\Lambda^{-q}[v_1,\dots,v_d].\]
		Since $\w H^{p-1}(\Delta;\Qq)$ has zero $\Qq[m]$-module degree, $\mathrm{^I}E_1$ collapses at the $E_1$-term.
		This implies that
		\begin{equation}\label{eq:E_1}
			(H^{-k}(\RR^*,D))_{2i}=\tbinom{d}{i}\w H^{i-k-1}(\Delta;\Qq).
		\end{equation}
		
		Since $\Delta$ is a Buchsbaum complex, Theorem \ref{thm:algebraic property} \eqref{item:3} implies that $\mathrm{st}_\sigma\Delta$ is Cohen-Macaulay for each $\sigma\neq\varnothing$. Hence, for the second spectral sequence,  we have
		\begin{gather*}
			\mathrm{^{II}}E_1^{p,-q}=0\quad \text{for } p,q>0,\quad \text{and}\\
			\mathrm{^{II}}E_1^{p,0}=\bigoplus_{\sigma\in\FF_{p-1}(\Delta)}\Qq[\mathrm{st}_\sigma\Delta]/\Theta\quad \text{for } p>0.
		\end{gather*}
		It follows that this spectral sequence collapses at the $E_2$-term, and so $H^k(\RR^*,D)=\bigoplus_{p-q=k}\mathrm{^{II}}E_2^{p,-q}.$
		The $E_1$-term also tells us that
		\begin{equation}\label{eq:E_2}
			\begin{split}
				\mathrm{^{II}}E_1^{0,-q}&=\mathrm{^{II}}E_2^{0,-q}\quad \text{for }q>0,\text{ and }\\
				\mathrm{^{II}}E_2^{0,-q}&=H^{-q}(\RR^*,D)\quad \text{for }q\geqslant 0.
			\end{split}
		\end{equation}
		
		An easy calculation shows that the $E_3$-term of the Serre spectral sequence in the lemma is just $E_3^{2p,q}=(\mathrm{^{II}}E_1^{0,-q})_{2p+2q}$.
		Combining this with formula \eqref{eq:E_1} and \eqref{eq:E_2}, we get the desired dimension of $E_3^{2p,q}$ for $q>0$.
		
		It remains to consider $E_3^{2p,0}$. Note that $\Ll_{2k}^*$ is a subcomplex of $\Ll^*$ since the internal grading is preserved by the differentials.
		Moreover, for each $k$ and $i$, $\Ll_{2k}^{-i}$ is a finite dimensional vector space over $\Qq$, so we can calculate the Euler characteristic of $\Ll_{2k}^*$:
		\[\chi(\Ll_{2k}^*)=\sum_{0\leqslant i\leqslant d}(-1)^i\dim_\Qq\Ll^{-i}_{2k}=\sum_{0\leqslant i\leqslant k}(-1)^i\dim_\Qq\Ll^{-i}_{2k}.\]
		Recall the Hilbert series of $\Qq[\Delta]$ is
		\[F(\Qq[\Delta],\lambda)=\frac{h_0+h_1\lambda^2+\cdots+h_d\lambda^{2d}}{(1-\lambda^2)^d}=1+a_1\lambda^2+a_2\lambda^4+\cdots\]
		It follows that $\chi(\Ll_{2k}^*)=\sum_{i=0}^k(-1)^i\tbinom{d}{i}a_{k-i}$. A straightforward calculation shows that this number is equal to the coefficient of $\lambda^{2k}$ in the expansion of the polynomial
		\[(1-\lambda^2)^d(1+a_1\lambda^2+a_2\lambda^4+\cdots),\] which is just $h_k$.
		
		On the other hand, we can also compute $\chi(\Ll_{2k}^*)$ in terms of the cohomology of $\Ll^*$, i.e.,
		\[\chi(\Ll_{2k}^*)=\sum_{0\leqslant i\leqslant k}(-1)^i\dim_\Qq H^{-i}(\Ll^*,d)_{2k}.\]
		Since $\chi(\Ll_{2k}^*)=h_k$ and $H^{-i}(\Ll^*,d)_{2k}=E_3^{2k-2i,i}$, we can immediately get the desired expression of $E_3^{2k,0}$ from the calculation of $E_3^{2k,>0}$ above.
	\end{proof}
	Now we can give the cohomology of toric spaces associated to Buchsbaum complexes.
	\begin{thm}\label{thm:coho of toric over buchs}
		Let $\Delta$ be a Buchsbaum complex, $\Theta$ an l.s.o.p. for $\Qq[\Delta]$. Then for the associated toric space $M_\Delta$, we have
		\[H^k(M_\Delta;\Qq)\cong\bigoplus_{2p+q=k} E_3^{2p,q},\]
		where $E_3^{2p,q}$ is given by Lemma \ref{lem:E3-page}. The ring structure of $H^k(M_\Delta;\Qq)$ is given by
		\[H^*(M_\Delta;\Qq)\cong \RR \oplus\Qq[\Delta]/\Theta,\quad \RR^k=\bigoplus_{q>0,\,2p+q=k}\tbinom{d}{p+q}\w H^{p-1}(\Delta;\Qq),\]
		where $\RR$ has trivial multiplication structure.
	\end{thm}
	\begin{proof}
		According to Theorem \ref{thm:cohomology of orbifold} and lemma \ref{lem:E3-page}, it is equivalent to prove that $E_3=E_\infty$ in the Serre spectral sequence of the fibration in Lemma \ref{lem:E3-page}.
		
		For each vertex $\{i\}\in[m]$, there is an inclusion $X_i=\ZZ_{\mathrm{st}_i\Delta}\times T^{m-j_i}\subset \ZZ_\Delta$, where $j_i=\#\FF_0(\mathrm{st}_i\Delta)$, and we have a Serre fibration map:
		\[\xymatrix{
			T^d \ar[r] \ar@{=}[d] & ET^m\times_{T^{m-d}}X_i \ar[r]\ar[d]^-{\varphi_i} & ET^{m}\times_{T^m}X_i \ar[d]\\
			T^d \ar[r]  & ET^m\times_{T^{m-d}}\ZZ_\Delta \ar[r] & ET^{m}\times_{T^m}\ZZ_\Delta
		}\]
		$\varphi_i$ induces a morphism of cohomology $\varphi_i^*:H^*(M_\Delta)\to H^*(M_i)$ (after an isomorphism given by Theorem \ref{thm:cohomology of orbifold}). Since the Serre spectral sequence construction is functorial, $\varphi_i$ also
		induces a morphism $E_r\xr{\varphi^*_{i,r}} \sideset{_i}{}{\mathop{E}}_r$ of Serre spectral sequences.
		
		Let $\mathrm{^{II}}E_r$ be the spectral sequence defined in the proof of Lemma \ref{lem:E3-page}, $\Phi_r=\bigoplus_{i=1}^m\varphi^*_{i,r}$. Then for the $E_3$-term of the Serre spectral sequence, the map $\Phi_3^{2p,q}$ is just the differential $d_1:\mathrm{^{II}}(E_1^{0,-q})_{2p}\to \mathrm{^{II}}(E_1^{1,-q})_{2p}$.
		So from \eqref{eq:E_1} and \eqref{eq:E_2}, we have
		\begin{equation}\label{eq:kernel}
			\begin{split}
				\mathrm{Ker}\,\Phi_3^{2p,q}&=E_3^{2p,q}\quad \text{for }q>0,\ \text{ and}\\
				\mathrm{Ker}\,\Phi_3^{2p,q}&=\mathrm{^{II}}(E_2^{0,-q})_{2p+2q}=\tbinom{d}{p+q}\w H^{p-1}(\Delta;\Qq)\quad \text{for }q\geqslant 0.
			\end{split}
		\end{equation}
		It follows that $\dim\mathrm{Ker}\,\Phi_3=\sum_{p,q=0}^{d}\tbinom{d}{p+q}\w \beta_{p-1}(\Delta)$.
		
		On the other hand, combining Lemma \ref{lem:zero mapping} and Corollary \ref{cor:cokernel} with the fact that $\Phi_\infty=\bigoplus_{i=1}^m\varphi^*_i$, we can get that
		\[\dim\mathrm{Ker}\,\Phi_\infty\geqslant \dim\mathrm{Im\,}j^* \geqslant\sum_{p,q=0}^{d}\tbinom{d}{p+q}\w \beta_{p-1}(\Delta).\]
		However, since $\mathrm{st}_i\Delta$ is Cohen-Macaulay,
		$\sideset{_i}{}{\mathop{E}}_3=\sideset{_i}{}{\mathop{E}}_\infty$, which implies that
		\[\dim\mathrm{Ker}\,\Phi_\infty\leqslant\dim\mathrm{Ker}\,\Phi_3=\sum_{p,q=0}^{d}\tbinom{d}{p+q}\w \beta_{p-1}(\Delta).\]
		Hence these two inequalities shows that $\mathrm{Ker}\,\Phi_3=\mathrm{Ker}\,\Phi_\infty$, and therefore $E_3^{2p,>0}$ survives to $E_\infty$  by formula \eqref{eq:kernel}, but this already implies that the Serre spectral sequence collapses at the $E_3$-term.
		
		It remains to see the ring structure of $H^*(M_\Delta;\Qq)$. It is clear that $E_3^{*,0}=\Qq[\Delta]/\Theta$. So we can define $\RR=E_3^{*,>0}$.
		Lemma \ref{lem:zero mapping} shows that $\mathrm{Ker}\,\Phi_3\subset \mathrm{Im}\,j^*$, and formula \eqref{eq:kernel} shows that $E_3^{*,>0}=\mathrm{Ker}\,\Phi_3^{*,>0}$.
		Thus the cohomology ring formula follows immediately from Lemma \ref{lem:product vanishing}.
	\end{proof}
	
	\begin{rem}\label{rem:ker=im}
		Note that the two inequalities in the proof of Theorem \ref{thm:coho of toric over buchs} also implies that $\mathrm{Ker}\,\Phi_3=\mathrm{Im}\,j^*$, where $j^*$ is the map in the exact sequence \eqref{eq:long exact}.
	\end{rem}
	From the proof of Theorem \ref{thm:coho of toric over buchs}, we can get an interesting result about the socle of a Buchsbaum complex over $\Qq$,
	which was initially obtained by Novik and Swartz \cite{NS09a} for any infinite field $\kk$.
	\begin{prop}[{\cite[Theorem 2.2]{NS09a}}]\label{prop:socle}
		Let $\Delta$ be a Buchsbaum complex, $\Theta$ an l.s.o.p. for $\Qq[\Delta]$.  Then
		\[\dim\mathrm{Soc}(\Qq[\Delta]/\Theta)_{2p}\geqslant\tbinom{d}{p}\w\beta_{p-1}(\Delta).\]
	\end{prop}
	\begin{proof}
		In formula \eqref{eq:kernel}, we have $\mathrm{Ker}\,\Phi_3^{2p,0}=\mathrm{^{II}}(E_2^{0,0})_{2p}=\tbinom{d}{p}\w H^{p-1}(\Delta;\Qq)$.
		Since $\mathrm{Ker}\,\Phi_3=\mathrm{Im}\,j^*$, by using Lemma \ref{lem:product vanishing} and Theorem \ref{thm:coho of toric over buchs} we get that $\mathrm{Ker}\,\Phi_3^{2p,0}\subset \mathrm{Soc}(\Qq[\Delta]/\Theta)_{2p}$, and so the inequality in the proposition holds.
	\end{proof}
	
	\subsection{Toric spaces associated to rational homology manifolds}
	In this subsection, $\Delta$ is a rational homology manifold without boundary. As we have shown, the toric space $M_\Delta$ is not a rational homology manifold unless $\Delta$ is a rational homology sphere. However, if we look at the local topology of $M_\Delta$ in the D-J construction, we can see that an open neighbourhood of a point $x\in M_\Delta-\{*\}\times T^d$, where $*$ is the cone point in $\CC\Delta$, is the same as the case that $\Delta$ is a rational homology sphere.
	So $M_\Delta-\{*\}\times T^d$ is an open manifold, and the subspace ${I\times\Delta\times T^d}/\sim\subset M_\Delta$ (see \eqref{eq:decompositon}) is a compact manifold with boundary $\{1\}\times\Delta\times T^d$.
	Moreover, if $\Delta$ is $\Qq$-orientable, then so is ${I\times\Delta\times T^d}/\sim$. The following result on its own is interesting in toric topology.
	
	\begin{thm}\label{thm:poincare duality}
		Let $\Delta$ be a $(d-1)$-dimensional connected simplicial complex. For $j^*$ in \eqref{eq:long exact}, let $\II=\bigoplus_{k=1}^{d-1}(\mathrm{Im}\,j^*)_{2k}$. If $\Delta$ is an
		orientable rational homology manifold without boundary, then the quotient algebra
		\[\Aa=H^*(M_\Delta;\Qq)/\II\]
		is a Poincar\'e duality algebra.
	\end{thm}
	\begin{proof}
		Since $\Delta$ is connected and orientable, it follows from Lemma \ref{lem:E3-page} and Theorem \ref{thm:coho of toric over buchs} that the top cohomology group of $M_\Delta$ is $H^{2d}(M_\Delta)=\Qq$. So it suffices to show that for any $0\neq\alpha\in \Aa_{2k}$ with $k<d$, there exists $\beta\in\Aa_{2d-2k}$, such that $\beta\alpha\neq0.$
		
		Let $M_1=\CC\Delta\times T^d$, $M_2={I\times\Delta\times T^d}/\sim$ in the decomposition formula \eqref{eq:decompositon}. Consider the following commutative diagram:
		\begin{equation}\label{eq:lefs dual diagram}
			\begin{gathered}
				\xymatrix{
					H^{2d-2k}(M_\Delta,M_1)\otimes H^{2k}(M_\Delta) \ar@<-6ex>[d]_-{\cong}^-{i^*} \ar@<10ex>[d]^-{i^*} \ar[r]^-{\ssm} & H^{2d}(M_\Delta,M_1) \ar[d]_-{\cong}^-{i^*}\\
					H^{2d-2k}(M_2,\partial M_2)\otimes H^{2k}(M_2)  \ar[r]^-{\ssm} & H^{2d}(M_2,\partial M_2)}
			\end{gathered}
		\end{equation}
		The vertical isomorphisms come from excision.
		
		Suppose $\alpha\in H^{2k}(M_\Delta)$ with $k<d$, such that its image is not zero in $\Aa_{2k}$, i.e., $\alpha\not\in \mathrm{Im}\,j^*$, then by the exactness of \eqref{eq:long exact}, $0\neq i^*(\alpha)\in H^{2k}(M_2)$.
		Since $M_2$ is an orientable manifold with boundary,
		the Lefschetz duality of $(M_2,\partial M_2)$ tells us that there exists $\beta\in H^{2d-2k}(M_2,\partial M_2)$ such that $\beta\ssm i^*(\alpha)\neq0$.
		Let $\beta'=(i^*)^{-1}(\beta)$.
		Diagram \eqref{eq:lefs dual diagram} shows that $\beta'\ssm \alpha\neq0$.
		Clearly the map $f^*:H^{2d}(M_\Delta,M_1)\to H^{2d}(M_\Delta)$ is an isomorphism. So $f^*(\beta')\ssm \alpha=f^*(\beta'\ssm \alpha)\neq 0$, and we get the desired element $f^*(\beta')$.
	\end{proof}
	\begin{rem}
		Suppose $[M_\Delta]$ is a generator of $H_{2d}(M_\Delta;\Qq)\cong\Qq$. Then for a face monomial $\xx_\sigma\in \Aa$, we have
		\[[M_\Delta]\sfr\xx_\sigma=\Qq\cdot(\rho_\sigma)_*([N_\sigma]),\]
		where $[N_\sigma]$ is the rational fundamental class of $N_\sigma$.
		This can be proved in the same way as Lemma \ref{lem:duality}.
	\end{rem}
	
	In \S \ref{subsec:buchsbaum} we have already seen that if $\Delta$ is a Buchsbaum complex, $\mathrm{Ker}\,\Phi_3=\mathrm{Im}\,j^*$, $\RR=\mathrm{Ker}\,\Phi^{*,>0}_3$ and $\mathrm{Soc}(\Qq[\Delta]/\Theta)_{2k}\supset \mathrm{Ker}\,\Phi^{2k,0}_3=\tbinom{d}{k}\w H^{k-1}(\Delta;\Qq)$. So Theorem \ref{thm:poincare duality} implies that if $\Delta$ is an orientable rational homology manifold, then
	\[\begin{split}
		&\dim\mathrm{Soc}(\Qq[\Delta]/\Theta)_{2k}=\tbinom{d}{k}\w\beta_{k-1}(\Delta),\\
		&\dim\Aa_{2k}=h_k(\Delta)-\tbinom{d}{k}\sum_{i=0}^{k-1}(-1)^i\w \beta_{k-i-1}(\Delta),\quad \text{and}\\
		&\Aa\cong \Qq[\Delta]/(\Theta+I),\quad\text{where }I=\bigoplus_{k=1}^{d-1}\mathrm{Soc}(\Qq[\Delta]/\Theta)_{2k}.
	\end{split}\]
	
	So in fields of characteristic zero, Theorem \ref{thm:poincare duality} is a topological explanation of the following important result of Novik and Swartz:
	\begin{thm}[\cite{NS09b}]\label{thm:socle of manifolds}
		Let $\Delta$ be a $(d-1)$-dimensional connected simplicial complex, $\kk$ an infinite field. If  $\Delta$ is an
		orientable $\kk$-homology manifold without boundary, then for any l.s.o.p. $\Theta$ for $\kk[\Delta]$.
		\[\dim\mathrm{Soc}(\kk[\Delta]/\Theta)_{2i}=\tbinom{d}{i}\w\beta_{i-1}(\Delta;\kk).\]
		Moreover, let $I=\bigoplus_{i=1}^{d-1}\mathrm{Soc}(\kk[\Delta]/\Theta)_{2i}$, then $\kk[\Delta]/(\Theta+I)$ is a Poincar\'e duality $\kk$-algebra.
	\end{thm}
	
	\appendix
	\section{Some topological facts about toric spaces}
	\subsection{Proof of Theorem \ref{thm:cohomology of orbifold}}\label{appdx:A1}
	The cohomology with coefficients in $\Qq$ will be implicit throughout the proof.
	First, we consider the case that $\mathit{\Lambda}:\Zz^m\to\Zz^d$ is onto. In this case, $M(\Delta,\mathit{\Lambda})=\ZZ_\Delta/T^{m-d}$, where $T^{m-d}$ is the kernel of the tori map $\exp\mathit{\Lambda}:T^m\to T^d$.
	Using the notation in subsection \ref{subsec:m-a}, we have a $T^m$-subspace $B_\sigma\subset \ZZ_\Delta$ for each $\sigma\in\Delta$.
	By definition, if $\dim\sigma=k-1$, $B_\sigma=(D^2)^k_\sigma\times T^{m-k}_{[m]\setminus\sigma}$. Let
	\[C_\sigma=\{(z_1,\dots,z_m)\in(D^2)^m:z_i=0 \text{ for } i\in\sigma\}.\]
	Then $C_\sigma\cong T^{m-k}$ is a $T^m$-invariant subspace of $B_\sigma$. Since $C_\sigma$ is a deformation retract of $B_\sigma$,
	$C_\sigma/T^{m-d}=T^{d-k}$ is a deformation retract of $B_\sigma/T^{m-d}$. It follows that the composition $T^{m-d}\to B_\sigma\to B_\sigma/T^{m-d}$ induces an isomorphism
	\[H^*(B_\sigma/T^{m-d})\otimes H^*(T^{m-d})\cong H^*(B_\sigma).\]
	
	On the other hand, applying the Leray-Hirsch theorem to the fiber bundle $T^{m-d}\to ET^m\times B_\sigma\to ET^m\times_{T^{m-d}}B_\sigma$, we get a $H^*(ET^m\times_{T^{m-d}}B_\sigma)$-module isomorphism:
	\[H^*(ET^m\times_{T^{m-d}}B_\sigma)\otimes H^*(T^{m-d})\cong H^*(ET^m\times B_\sigma).\]
	Hence from a commutative diagram argument it follows that there is a ring isomorphism
	\[p_\sigma^*:H^*(B_\sigma/T^{m-d})\to H^*(ET^m\times_{T^{m-d}}B_\sigma),\]
	which is induced by the restriction $p_\sigma:ET^m\times_{T^{m-d}}B_\sigma\to B_\sigma/T^{m-d}$ of $p$.
	
	Now we can get the desired cohomology isomorphism by double induction on the number of facets of $\Delta$ and $\dim\Delta$.
	We proceed the inductive argument by applying Mayer-Vietoris sequences, and the base of the induction is given above.
	
	For the general case that $M(\Delta,\mathit{\Lambda})=\ZZ_\Delta/(T^{m-d}\times G)$, consider the quotient map
	$\pi:\ZZ_\Delta/T^{m-d}\to M(\Delta,\mathit{\Lambda})$.
	We only need to show that $\pi^*:H^*(M(\Delta,\mathit{\Lambda}))\to H^*(\ZZ_\Delta/T^{m-d})$ is an isomorphism.
	Note that the $G$-action on $\ZZ_\Delta/T^{m-d}$ extends to a toral action. Thus, $H^*(\ZZ_\Delta/T^{m-d})$ is fixed under the induced $G$-action.
	Recall the classical result for a finite $G$-action: $G\times X\to X$, that is
	\[\pi^*:H^*(X/G;\Qq)\to H^*(X;\Qq)^G:=\{x\in H^*(X;\Qq):gx=x\}\]
	is a ring isomorphism (see for example \cite{Bor60}). Then the theorem follows. \qed
	\subsection{Algebraic model for cellular cochains}\label{appdx:cell rep}
	In this subsection, we prove the promised statement in the proof of Lemma \ref{lem:duality}. That is, for a $(k-1)$-face $\sigma\in\Delta$,
	the monomial $\xx_\sigma\in\Qq[\Delta]/\Theta=H^*(M_\Delta;\Qq)$ is, up to multiplication by an integer, represented by a cocycle $\tilde e^*_\sigma\in C^{2k}(M_\Delta;\Qq)$. First, we need to know the algebraic models for the cellular cochain algebras $C^*(\ZZ_\Delta)$, $C^*(ET^m)$, $C^*(BT^m)$, $C^*(ET^m\times_{T^m}\ZZ_\Delta)$, etc.
	
	Recall that in \cite[\S 4.5]{BP15}, $S^\infty$ is given a cell decomposition with one cell in each dimension; the boundary of an even cell is the
	closure of an odd cell, and the boundary of an odd cell is the $0$-cell. Thus, the cellular cochain complex of $S^\infty$ can be identified with
	the Koszul algebra
	\[\Lambda[y]\otimes\Zz[u],\quad \mathrm{deg}\,y=1,\ \mathrm{deg}\,u=2,\ dy=u,\ du=0.\]
	Similarly, $\Cc P^\infty$ has a cell decomposition with one cell in each even dimension, and $C^*(\Cc P^\infty)$ can be identified with the polynomial algebra $\Zz[u]$. It follows that the cochain homomorphism $C^*(BT^m)\to C^*(ET^m)$ induced by the universal principal $T^m$-bundle $ET^m\to BT^m$ has an algebraic model of the form
	\[\Zz[u_1,\dots,u_m]\to\Lambda[y_1\dots,y_m]\otimes\Zz[u_1,\dots,u_m].\]
	
	On the other hand, recall the Koszul complex $(\Lambda[y_1\dots,y_m]\otimes\Zz[\Delta],d)$ of the face ring  defined in subsection \ref{subsec:cohom m-a}.
	It is an algebraic model of $C^*(\ZZ_\Delta)$. Precisely, the cochain map
	\[\Lambda[y_1\dots,y_m]\otimes\Zz[\Delta]\to C^*(\ZZ_\Delta),\quad y_i\mapsto t_i^*;\ x_i\mapsto e_i^*\]
	induces an isomorphism in cohomology.
	It follows that the differential graded algebra
	\[
	\RR=\Zz[u_1,\dots,u_m]\otimes\Lambda[y_1\dots,y_m]\otimes\Zz[\Delta],\quad dy_i=u_i-x_i
	\]
	is an algebraic model of $C^*(ET^m\times_{T^m}\ZZ_\Delta)$ (see for example \cite[Theorem 12.6.1]{Nei10}).
	This has the consequence that for a $(k-1)$-face $\sigma\in\Delta$, the monomial $\xx_\sigma\in\Zz[\Delta]=H^*(ET^m\times_{T^m}\ZZ_\Delta)$ can be represented by the cocycle \[(pt, e_\sigma)^*\in C^{2k}(ET^m\times_{T^m}\ZZ_\Delta).\]
	
	Now consider the fibration sequence
	\[ET^{m-d}\times_{T^{m-d}}\ZZ_\Delta\xr{i} ET^m\times_{T^m}\ZZ_\Delta\xr{\pi} BT^d.\]
	From the proof of Proposition \ref{prop:equi toric} we know that the fiber inclusion map $i$ induces a homomorphism of rational cohomology $i^*:\Qq[\Delta]\to\Qq[\Delta]/\Theta$. Since $i^*((pt, e_\sigma)^*)=(pt, e_\sigma)^*$,  $\xx_\sigma\in \Qq[\Delta]/\Theta=H^*(ET^{m-d}\times_{T^{m-d}}\ZZ_\Delta;\Qq)$ can also be represented by the cocycle \[(pt, e_\sigma)^*\in C^{2k}(ET^{m-d}\times_{T^{m-d}}\ZZ_\Delta;\Qq).\]
	
	On the other hand, for the `orbit cell' $\tilde e_\sigma$, the cellular cochain $\tilde e_\sigma^*\in C^{2k}(M_\Delta)$ satisfies that
	$p^*(\tilde e^*_\sigma)=(pt, e_\sigma)^*$ up to  multiplication by an integer, where $p:ET^{m-d}\times_{T^{m-d}}\ZZ_\Delta\to M_\Delta$ is the quotient map. Since $p^*$ is an isomorphism on rational cohomology by Theorem \ref{thm:cohomology of orbifold},
	the assertion follows.
	\subsection{Proof of Lemma \ref{lem:cokernel}}\label{appdx:proof of cokernel lemma}
	Let $\Delta'$ be the barycentric subdivision of $\Delta$. Recall that the relation `$\sim$' in $\Delta\times T^d/\sim$ is defined by means of the polyhedral decomposition
	\[\Delta'=\bigcup_{i\in\FF_0(\Delta)}\mathrm{st}_i\Delta'.\]
	Thus, $\mathscr{U}=\{\UU_i:={\mathrm{st}_i\Delta'\times T^d}/\sim\}_{i\in\FF_0(\Delta)}$ can be viewed as an `open' cover of ${\Delta\times T^d}/\sim$.
	For a subset $\sigma=\{i_1,\dots,i_k\}\subset[m]$, let $\UU_\sigma$ denote the intersection $\UU_{i_1}\cap\cdots\cap\UU_{i_k}$. There are some obvious facts:
	
	(i) $\UU_\sigma\neq\varnothing$ if and only if $\sigma\in\Delta$.
	
	(ii) For $\sigma\in\Delta$, $\UU_\sigma$ is a toric space associated to the geometric realization of the
	poset $\Delta_{>\sigma}=\{\tau\in K:\tau>\sigma\}$. This geometric realization, which we denote by $L_\sigma$, is a subcomplex of $\Delta'$ and combinatorially equivalent to $\mathrm{lk}_\sigma\Delta$. Precisely,
	\[\UU_\sigma={F_\sigma \times T^{d-|\sigma|}}/\sim,\]
	where $F_\sigma\subset \Delta'$ is the geometric realization of the
	poset $\Delta_{\geqslant\sigma}$ as defined in \S \ref{subsec:m-a},  $T^{d-|\sigma|}=T^d/(T_{i_1}\times\cdots\times T_{i_k})$.
	
	Now consider the \v Cech double complex
	\[
	\begin{split}
		(\KK^*,d)=\bigoplus_k\KK^k=\bigoplus_k\bigoplus_{p+q=k}(\KK^{p,q},\partial,\delta),\\
		\KK^{p,q}=\bigoplus_{\sigma\in\FF_p(\Delta)}C^q(\UU_\sigma,\partial;\Qq),\quad d=\delta+(-1)^p\partial,
	\end{split}
	\]
	where $C^*(\UU_\sigma,\partial;\Qq)$ is the rational cellular cochain complex of $\UU_\sigma$ and $\delta$ is the \v  Cech coboundary operator.
	There are two spectral sequences converging to the total cohomology $H^*(\KK^*,d)$. One spectral
	sequence starts with $\mathrm{^I}E_1=H_\delta$ and $\mathrm{^I}E_2=H_\partial H_\delta$, and another with $\mathrm{^{II}}E_1=H_\partial$ and $\mathrm{^{II}}E_2=H_\delta H_\partial $. (The second one is also known as the \emph{Mayer-Vietoris spectral Sequence}.)
	
	Since $\mathscr{U}$ is an open cover, \[\mathrm{^{I}}E_1=\mathrm{^{I}}E_1^{0,*}=C^*({\Delta\times T^d}/\sim,\partial;\Qq).\]
	Hence, the first spectral sequence collapses at the $E_1$-term and therefore
	\[H^*(\KK^*,d)\cong H^*({\Delta\times T^d}/\sim;\Qq).\]
	For the second spectral sequence, we have $\mathrm{^{II}}E_1^{p,q}=0$ for $q$ odd, and
	\[\mathrm{^{II}}E_1^{p,2q}=\bigoplus_{\sigma\in\FF_p}H^{2q}(\UU_\sigma;\Qq).\]
	This is because $\UU_\sigma$ is a toric space associated ot $L_\sigma$, which is a Cohen-Macaulay complex by the assumption that $\Delta$ is a Buchsbaum complex.
	
	It follows from Appendix \ref{appdx:cell rep} that
	a cohomology class of $H^{2q}(\UU_\sigma;\Qq)$ can be represented by the cellular cochain $\sum k_\varrho\tilde  e^*_\varrho,$
	where $\varrho\in L_\sigma$ is a simplex of the form $\varrho =(\sigma_1<\cdots<\sigma_q)$ with $\sigma_i\in\Delta_{>\sigma}$, $\tilde e^*_\varrho$ is the `orbit cell' corresponding to the cell $e_\varrho=e^2_{\sigma_1}\times\cdots\times e^2_{\sigma_q}\subset\ZZ_{L_\sigma}$.
	An easy topological observation shows that $e_\varrho$ is the image of $(\sigma<\sigma_1<\cdots<\sigma_q)\times T^q_\varrho$ under the quotient map $F_\sigma\times T^k\to\ZZ_{L_\sigma}$, where $k=|\FF_0(L_\sigma)|$. Hence,
	\[\tilde e_\varrho=\pi((\sigma<\sigma_1<\cdots<\sigma_q)\times T^q_\varrho),\quad\text{where } \pi:F_\sigma\times T^{d-|\sigma|}\to\UU_\sigma.\]
	
	Since $\mathrm{^{II}}E_\infty\cong H^*({\Delta\times T^d}/\sim;\Qq)$, and the analysis  above shows that an element $\alpha\in \mathrm{^{II}}E_\infty^{p,2q}$ is represented by the linear combination of the dual orbit cells $\tilde e^*_\varrho$ with $\varrho\in\FF_{q-1}(L_\sigma)$ for some $\sigma\in\FF_p(\Delta)$, a diagram chasing in the double  complex $\KK^{*,*}$  shows that $\alpha$ is represented by a cocycle $\beta\in C^{p+2q}({\Delta'\times T^d}/\sim)$ of the form
	\[\begin{split}
		\beta&=\sum_{\tau\in\FF_{p+q}(\Delta')}k_\tau\tilde c^*_\tau,\quad \text{where $\tilde c_\tau\subset {\Delta'\times T^d}/\sim$ has the form}\\
		\tilde c_\tau&=\pi((\tau_1<\cdots<\tau_p<\sigma<\sigma_1\cdots<\sigma_q)\times T^q_{\varrho}).
	\end{split}\]
	So $\pi^*(\beta)\in H^{p+q}(\Delta';\Qq)\otimes H^q(T^d;\Qq)$, and therefore lemma \ref{lem:cokernel} holds. \qed
	
	\bibliography{M-A}

\providecommand{\bysame}{\leavevmode\hbox to3em{\hrulefill}\thinspace}
\providecommand{\MR}{\relax\ifhmode\unskip\space\fi MR }
\providecommand{\MRhref}[2]{%
  \href{http://www.ams.org/mathscinet-getitem?mr=#1}{#2}
}
\providecommand{\href}[2]{#2}
\begin{thebibliography}{10}

\bibitem{A18}
K.~Adiprasito, \emph{\textrm{Combinatorial Lefschetz theorems beyond
  positivity}}, arXiv:1812.10454, 2018.

\bibitem{BN10}
E.~Babson and E.~Nevo, \emph{Lefschetz properties and basic constructions on
  simplicial spheres}, J. Algebraic Combin. \textbf{31} (2010), no.~1,
  111--129.

\bibitem{BBP04}
I.~Baskakov, V.~Buchstaber, and T.~Panov, \emph{Cellular cochain algebras and
  torus actions}, Russian Math. Surveys \textbf{59} (2004), no.~3, 562--563.

\bibitem{B02}
I.~V. Baskakov, \emph{Cohomology of \textrm{$K$}-powers of spaces and the
  combinatorics of simplicial divisions}, Uspekhi Mat. Nauk \textbf{57} (2002),
  no.~5, 147--148 (Russian), Russian Math. Surveys, \textbf{57}(2002), no. 5,
  898-990 (English translation).

\bibitem{BS15}
J.~B\"ohm and S.~A. Papadakis, \emph{Weak \text{L}efschetz property and stellar
  subdivisions of \text{G}orenstein complexes}, arXiv:1501.01513, 2015.

\bibitem{Bor60}
A.~Borel, G.~Bredon, E.~E. Floyd, D.~Montgomery, and R.~Palais, \emph{Seminar
  on transformation groups.}, Annals of Mathematics Studies, vol.~46, Princeton
  University Press, Princeton, New Jersey, 1960.

\bibitem{BH98}
W.~Bruns and J.~Herzog, \emph{\textrm{Cohen-Macaulay Rings}}, revised ed.,
  Cambridge Studies in Adv. Math., vol.~39, Cambridge Univ. Press, Cambridge,
  1998.

\bibitem{BP00}
V.~Buchstaber and T.~Panov, \emph{Torus actions, combinatorial topology and
  homological algebra}, Russian Math. Surveys \textbf{55} (2000), no.~5,
  825--921.

\bibitem{BP02}
\bysame, \emph{Torus actions and their applications in topology and
  combinatorics}, University Lecture Series, vol.~24, Amer. Math. Soc.,
  Providence, RI, 2002.

\bibitem{BP15}
\bysame, \emph{Toric topology}, Mathematical Surveys and Monographs, vol. 204,
  Amer. Math. Soc., Providence, RI, 2015.

\bibitem{Buc08}
W.~Buczynska, \emph{Fake weighted projective spaces}, arXiv:0805.1211, 2008.

\bibitem{C17}
L.~Cai, \emph{On products in a real moment-angle manifold}, J. Math. Soc. Japan
  \textbf{69} (2017), no.~2, 503--528.

\bibitem{CLS11}
D.~Cox, J.~Little, and H.~Schenck, \emph{Toric varieties}, Graduate Studies in
  Mathematics, vol. 124, American Mathematical Society, Providence, RI, 2011.

\bibitem{Dan78}
V.~Danilov, \emph{The geometry of toric varieties,}, Uspekhi Mat. Nauk.
  \textbf{33} (1978), no.~2, 85--134, English translation, Russian Math.
  Surveys \textbf{33} (1978), 97-154.

\bibitem{DJ91}
M.~W. Davis and T.~Januszkiewicz, \emph{Convex polytopes, coxeter orbifolds and
  torus actions}, Duke Math. J. \textbf{62} (1991), no.~2, 417--451.

\bibitem{FW15}
F.~Fan and X.~Wang, \emph{On the cohomology of moment-angle complexes
  associated to \text{G}orenstein* complexes}, arXiv:1508.00159, 2015.

\bibitem{H75}
M.~Hochster, \emph{Cohen-\text{M}acaulay rings, combinatorics, and simplicial
  complexes}, Ring Theory II (Proc. Second Oklahoma Conf.), Lect. Notes Pure
  Appl. Math., vol.~26, Dekker, New York, 1977, pp.~171--233.

\bibitem{Kaw73}
T.~Kawasaki, \emph{Cohomology of twisted projective spaces and lens complexes},
  Math. Ann. \textbf{206} (1973), 243--248.

\bibitem{Kle64}
V.~Klee, \emph{A combinatorial analogue of poincar\'e's duality theorem},
  Canad. J. Math. \textbf{16} (1964), 517--531.

\bibitem{Mac27}
F.~Macaulay, \emph{Some properties of enumeration in the theory of modular
  systems}, Proc. London Math. Soc. \textbf{26} (1927), 531--555.

\bibitem{Mc71}
P.~McMullen, \emph{The number of faces of simplicial polytopes}, Israel J.
  Math. \textbf{9} (1971), 559--570.

\bibitem{Mc93}
\bysame, \emph{On simple polytopes}, Invent. Math. \textbf{113} (1993), no.~2,
  419--444.

\bibitem{Mc96}
\bysame, \emph{Weights on polytopes}, Disc. Comp. Geom. \textbf{15} (1996),
  no.~4, 363--388.

\bibitem{MY14}
S.~Murai and K.~Yanagawa, \emph{\textrm{Squarefree $P$-modules and the
  $cd$-index}}, Adv. Math. \textbf{265} (2014), 241--279.

\bibitem{Nei10}
J.~Neisendorfer, \emph{Algebraic methods in unstable homotopy theory}, New
  Mathematical Monographs, vol.~12, Cambridge Univ. Press, Cambridge, 2010.

\bibitem{NS09b}
I.~Novik and E.~Swartz, \emph{\textrm{Gorenstein rings through face rings of
  manifolds}}, Compos. Math. \textbf{146} (2009), no.~4, 993--1000.

\bibitem{NS09a}
\bysame, \emph{\textrm{Socles of Buchsbaum modules, complexes and posets}},
  Adv. Math. \textbf{222} (2009), no.~6, 2059--2084.

\bibitem{O88}
T.~Oda, \emph{\textrm{Convex Bodies and Algebraic Geometry. An introduction to
  the theory of toric varieties}}, Ergeb. Math. Grenzgeb. (3), vol.~15,
  Springer-Verlag, Berlin, 1988.

\bibitem{OF70}
P.~Orlik and F.~Raymond, \emph{Actions of the torus on 4-manifolds}, Trans.
  Amer. Math. Soc. \textbf{152} (1970), no.~2, 531--559.

\bibitem{Rei76}
G.~Reisner, \emph{\textrm{Cohen-Macaulay} quotients of polynomial rings}, Adv.
  Math. \textbf{21} (1976), no.~1, 30--49.

\bibitem{Sch81}
P.~Schenzel, \emph{\textrm{On the number of faces of simplicial complexes and
  the purity of Frobenius}}, Math. Z. \textbf{178} (1981), 125--142.

\bibitem{S80}
R.~Stanley, \emph{The number of faces of a simplicial convex polytope}, Adv. in
  Math. \textbf{80} (1980), no.~3, 251--258.

\bibitem{S96}
\bysame, \emph{\textrm{Combinatorics and Commutative Algebra}}, 2nd ed.,
  Progress in Math., vol.~41, Birkhauser, Boston, 1996.

\bibitem{Swa06}
E.~Swartz, \emph{$g$-elements, finite buildings and higher
  \text{C}ohen-\text{M}acaulay connectivity}, J. Combin. Theory Ser. A
  \textbf{113} (2006), no.~7, 1305--1320.

\bibitem{Sw14}
\bysame, \emph{Thirty-five years and counting}, arXiv:1411.0987, 2014.

\end{thebibliography}
	\bibliographystyle{amsplain}
	
\end{document}